\documentclass{article}

\usepackage[left=1.2in,right=1.2in,top=1.5in,bottom=2in]{geometry}

\usepackage{amsmath}
\usepackage{amsfonts}  
\usepackage{amsthm}

\usepackage{authblk}

\usepackage{subfig}
\usepackage{booktabs}
\usepackage{enumitem}
\usepackage{url}

\usepackage{pdflscape}  
\usepackage{rotating}  

\usepackage{units}  

\usepackage{tikz}
\usetikzlibrary{calc,shapes,arrows,patterns}

\newcommand{\C}{\mathcal{C}}  
\newcommand{\D}{\mathcal{D}}  
\newcommand{\R}{\mathbb{R}}  
\newcommand{\Z}{\mathbb{Z}}  

\newcommand{\K}{\mathcal{K}}  
\newcommand{\Kb}{\mathcal{K}_{b}}  
\newcommand{\Kx}{\mathcal{K}_{x}}  
\newcommand{\Q}{\mathcal{Q}}  
\newcommand{\F}{\mathcal{F}}  

\newcommand{\X}{\mathcal{X}}  
\newcommand{\Y}{\mathcal{Y}}  

\newcommand{\MICONV}{MI-CONV}
\newcommand{\MILP}{MILP}
\newcommand{\MICONIC}{MI-CONIC}

\newtheorem{theorem}{Theorem}
\newtheorem{example}{Example}
\newtheorem{lemma}{Lemma}

\DeclareMathOperator{\cl}{cl}       
\DeclareMathOperator{\intr}{int}    
\DeclareMathOperator{\conv}{conv}   
\DeclareMathOperator{\opt}{opt}
\newcommand{\polar}[1]{#1^{\#}}     
\newcommand{\Kgt}[1]{\succ_{#1}}

\newcommand{\Kgeq}[1]{\succeq_{#1}}
\newcommand{\Kleq}[1]{\preceq_{#1}}
\newcommand{\knorm}[2]{\left| #2 \right|_{#1}}

\title{Disjunctive cuts in Mixed-Integer Conic Optimization}
\author[1]{Andrea Lodi}
\author[1]{Mathieu Tanneau\thanks{Corresponding author: mathieu.tanneau@polymtl.ca}}
\author[2]{Juan-Pablo Vielma}

\affil[1]{CERC Data Science, Polytechnique Montréal}
\affil[2]{massuchusetts Institute of Technology}

\begin{document}

\maketitle

\begin{abstract}

    This paper studies disjunctive cutting planes in Mixed-Integer Conic Programming.
    Building on conic duality, we formulate a cut-generating conic program for separating disjunctive cuts, and investigate the impact of the normalization condition on its resolution.
    In particular, we show that a careful selection of normalization guarantees its solvability and conic strong duality.
    Then, we highlight the shortcomings of separating conic-infeasible points in an outer-approximation context, and propose conic extensions to the classical lifting and monoidal strengthening procedures.
    Finally, we assess the computational behavior of various normalization conditions in terms of gap closed, computing time and cut sparsity.
    In the process, we show that our approach is competitive with the internal lift-and-project cuts of a state-of-the-art solver.
    
\end{abstract}

\section{Introduction}

Mixed-Integer Convex Optimization (\MICONV) is a fundamental class of Mixed-Integer Non-Linear Optimization problems with applications such as risk management, non-linear physics (e.g., power systems and chemical engineering) and logistics, just to mention a few. 
Because of such a relevance, classical algorithms for Mixed-Integer Linear Optimization (\MILP) have been successfully extended to \MICONV, like Branch and Bound \cite{bonami2008algorithmic} or Benders decomposition \cite{Geoffrion1970}; others like the Outer Approximation scheme \cite{Duran1986} have been designed specifically for \MICONV.
In addition, several software tools are available for solving general \MICONV\ problems, see, e.g., the recent comparison in \cite{Kronqvist2019}.
Finally, some specific classes of \MICONV\ problems, like Mixed-Integer (Convex) Quadratically Constrained Quadratic Optimization (MIQCQP) problems are now supported by the major commercial solvers.

Conic optimization is viewed as a more numerically stable and tractable alternative to general convex optimization \cite{Ben-Tal2001}.
Both classes are equivalent: conic optimization problems are convex, and any convex optimization problem can be written as a conic optimization problem \cite{KilincKarzan2015_MinimalValidInequalities}.
Modeling tools such as disciplined convex optimization \cite{Grant2006} can provide conic formulations for most --if not all-- convex optimization problems that arise in practice \cite{Mosek2020_ModelingCookbook}.
In particular, \cite{Lubin2018} recently showed that all convex instances in MINLPLib can be formulated as Mixed-Integer Conic Optimization (\MICONIC) problems using only a handful of cones.

Nevertheless, the intrinsic difference between convex and conic optimization lies in a problem's algebraic description: in the former, constraints are formulated as $f(x) \leq 0$, where $f$ is a convex function, whereas, in the latter, they are expressed using conic inequalities of the form $Ax -b \in \K$, where $A$ is a matrix, $b$ is a vector and $\K$ is a cone (see \cite{Ben-Tal2001} and Section \ref{sec:background}).
In particular, conic formulations enable the use of conic duality theory, which underlies a number of theoretical insights and practical tools.
Major commercial solvers have supported Mixed-Integer Second Order Cone Programming (MISOCP) for some time, and more general \MICONIC\ problems are now supported by a number of solvers, e.g., Mosek and Pajarito \cite{coey2018outer,Extended-Formulations-in-Mixed-Integer-Convex-CONF,Lubin2018}.

This paper builds on two specific aspects that we consider fundamental for solving \MICONV\ problems.
First, given that cutting planes are instrumental to solving \MILP, a number of authors have looked at various approaches to compute cuts for \MICONV\ problems and, nowadays, linear cutting planes are part of the arsenal of some \MICONV\ solvers.
Despite this (partial) success, some fundamental questions in this area are left unanswered.
Second, recent experience has shown that \emph{conic} formulations of \MICONV\ problems display enviable properties
that make them preferable, from the solving viewpoint, to generic \MICONV\ formulations \cite{Lubin2018,coey2018outer}.

In that context, motivated by the success of disjunctive cuts in \MILP, the paper focuses on computational aspects of disjunctive cuts for \MICONIC\ problems.
In addition, we answer the (somehow) natural question of what one can gain in terms of cutting planes by using a problem's \emph{conic} structure, as well as several questions left open by previous works on the topic.
In the remainder of this section, we review the literature on the subject and outline our main contributions.

\subsection{Disjunctive cuts: the MILP case}

    Disjunctive cuts in MILP date back to Balas' seminal work on disjunctive programming \cite{Balas1979} in the 70s, and became widely popular as their integration into branch-and-cut frameworks \cite{BalasEtAl1993_LiftProjectCutting,BalasEtAl1996_Mixed01} proved effective.
    Remarkably, disjunctive cuts, split cuts in particular, encompass several classes of cutting planes, e.g., Chvatal-Gomory, Gomory Mixed-Integer and Mixed-Integer Rounding cuts.
    
    A general approach for separating disjunctive cuts in MILP is the so-called Cut-Generating Linear Program (CGLP) proposed by Balas \cite{Balas1979,BalasEtAl1993_LiftProjectCutting}.
    The CGLP leverages a characterization of valid inequalities for disjunctive sets using Farkas multipliers, see Theorem 3.1 in \cite{Balas1979}.
    Thus, it is formulated in a higher-dimensional space, whose size is proportional to the number of disjunctive terms: for split cuts, which are two-term disjunctions, the CGLP is roughly double the size of the original problem.
    
    Computational aspects of the CGLP have been studied extensively, some of which we mention here.
    Given a fractional point $\bar{x}$ to separate, one can project the CGLP onto the support of $\bar{x}$, thereby reducing its size, and recover a valid cut by lifting \cite{BalasEtAl1993_LiftProjectCutting,BalasEtAl1996_Mixed01}.
    Split cuts obtained from solving the CGLP can be improved upon using monoidal strengthening \cite{BalasJeroslow1980_StrengtheningCutsMixed,BalasEtAl1996_Mixed01}.
    The normalization condition in the CGLP has been shown to have a major impact on the quality of the obtained cuts, and on overall performance \cite{Fischetti2011,Bonami2012_OptimizingOverLift,CadouxLemarechal2013_ReflectionsGeneratingdisjunctive,Serra2020_ReformulatingDisjunctiveCut,ConfortiWolsey2019_FacetSeparationOne}.
    In particular, Balas and Perregard \cite{Balas2003}, and later Bonami \cite{Bonami2012_OptimizingOverLift}, show that, in the case of split disjunctions, the CGLP can in fact be solved in the space of orignal variables only, yielding substantial computational gains.
    Recent developments include the efficient separation of cuts from multiple disjunctions \cite{Perregaard2001,Kazachkov2018_NonRecursiveCut}.

\subsection{Disjunctive cuts: the \MICONV\ case}

The work on disjunctive cutting planes for \MICONV\ (re)started already in the late 90s with two fundamental contributions \cite{Ceria1999,Stubbs1999}.
More precisely, Ceria and Soares \cite{Ceria1999} show that disjunctive convex problems can be formulated as a single convex problem in a higher dimensional space, and hint that this could serve to generate cutting planes using sub-gradient information at the optimum.
Around the same time, Stubbs and Mehrotra \cite{Stubbs1999} make the separation of disjunctive cuts for \MICONV\ explicit by \emph{(i)} solving one Non-Linear Programming (NLP) problem, and \emph{(ii)} identifying a sub-gradient that yields a violated cut.
The latter is done by taking a gradient (under regularity assumptions), or by solving a linear system (under the assumption that the objective function of the former problem is polyhedral).
Those assumptions and the use of perspective functions lead to differentiability issues that made the results of the computational investigation in \cite{Stubbs1999} numerically disappointing (according to the authors themselves). 

The numerical difficulties encountered in \cite{Stubbs1999} have slowed down the development of the area for a number of years --with the exception of \cite{Zhu2006}-- until the renewed interest and the practical approaches of the last decade \cite{Bonami2011,Kilinc2017}.
More precisely, Kilinc et al. \cite{Kilinc2017} note that 
\emph{``A simple strategy for generating lift-and-project cuts for a MINLP problem is to solve a CGLP
[...] based on a given polyhedral outer approximation of the relaxed feasible region [...].
The key question to be answered [...] is which points to use to define the polyhedral relaxation."} (from \cite{Kilinc2017}, Sec. 3). 

The distinction in how to answer the above question is the difference between \cite{Zhu2006}, \cite{Bonami2011}, and \cite{Kilinc2017}.
Namely, Zhu and Kuno \cite{Zhu2006} build an outer-approximation through the current fractional solution, and derive a cut by solving the associated CGLP.
However, this approach is not guaranteed to find a violated cut if one exists, see Example 1 in \cite{Kilinc2017}.
Bonami \cite{Bonami2011} solves one auxiliary NLP, and uses the solution to get an outer approximation that provably yields a violated cut if any exists.
Instead, Kilinc et al. \cite{Kilinc2017} iteratively refine an outer approximation by solving a sequence of LPs until a violated cut, if any, is separated by solving the associated CGLP.

In a recent paper, Kronqvist and Misener \cite{KronqvistMisener2020_DisjunctiveCutStrengthening} present a disjunctive-based cut-strengthening technique for \MICONV.
Given an initial valid inequality, and an ``exclusive" selection constraint (i.e., $x_{1} + ... + x_{k} = 1,~x \in \{0,1\}^k$), 
the procedure solves $k$ convex problems to tighten the cut's right-hand side and the coefficients of the $k$ binary variables.
This approach relates to the generalized disjunctive programming framework \cite{TrespalaciosGrossmann2014_ReviewMixedInteger}, for which cutting-plane algorithms based on \cite{Stubbs1999} have been proposed, see, e.g., \cite{TrespalaciosGrossmann2016_CuttingPlaneAlgorithm}.

The outer approximation approaches in \cite{Bonami2011,Kilinc2017} are, to the best of our knowledge, the state of the art for the implementation of disjunctive cuts for \MICONV\ and, especially, for MIQCQPs, see e.g., their implementation in CPLEX starting from version 12.6.2.
However, despite the impressive practical improvements with respect to the early attempts \cite{Stubbs1999}, questions were left on the table, which we answer in the present paper.

\subsection{Disjunctive cuts: the \MICONIC\ case}

    Following the support of MISOCP problems by major commercial solvers in the 2000s (MISOCP support appeared in CPLEX 9.0 and in Gurobi 5.0), the last decade has seen a flourishing literature on cuts for \MICONIC\ problems.
    
    A large share of these works focus on cuts for MISOCP, or, equivalently, for convex MIQCQP problems.
    Atamturk and Narayanan \cite{AtamturkNarayanan2010_ConicMixedInteger} introduce conic Mixed-Integer Rounding (MIR) cuts for MISOCP problems.
    Modaresi et al later show in \cite{Modaresi2015,Modaresi2016_ValidInequalitiesReformulation} that conic MIR cuts are in fact linear split cuts in an extended space, and compare the strength of families of conic MIR cuts to that of non-linear split cuts.
    In a related work, Andersen and Jensen \cite{Andersen2013} study intersection cuts in the MISOCP context, and obtain a closed-form formula for the conic quadratic intersection cut.
    Belotti et al. \cite{belotti2015conic} study  the intersection of a convex set and a two-term disjunction.
    They show that the convex hull is described by a single conic inequality, for which an explicit formula is derived in the conic quadratic case.
    In a similar fashion, two-term disjunctions on the second-order cone are investigated in \cite{KilincKarzanYildiz2015_TwoTermDisjunctions}, and this approach is later extended in \cite{YildizCornuejols2015_DisjunctiveCutsCross}.

    More general approaches, i.e., not restricted to convex quadratic constraints, include \cite{Cezik2005,AtamturkNarayanan2011_LiftingConicMixed,DadushEtAl2011_SplitClosureStrictly,KilincKarzan2015_MinimalValidInequalities,ModaresiEtAl2016_IntersectionCutsNonlinear}.
    In \cite{Cezik2005}, the authors study classes of cutting planes in the \MICONIC\ setting, including Chvatal-Gomory cuts and lift-and-project cuts, and report limited experiments on mixed 0-1 semi-definite programming instances.
    A generic lifting procedure for conic cuts is described in \cite{AtamturkNarayanan2011_LiftingConicMixed}.
    Dadush et al. \cite{DadushEtAl2011_SplitClosureStrictly} show that the split closure of a strictly convex body is defined by a finite number of disjunction, but is not necessarily polyhedral.
    Minimal valid inequalities are introduced in \cite{KilincKarzan2015_MinimalValidInequalities}, and are shown to be sufficient to describe the convex hull of a disjunctive conic set.
    The lack of tractable algebraic representation for minimal inequalities then leads the author to consider the broader class of sublinear inequalities, which are further studied in \cite{KilincKarzanSteffy2016_SublinearInequalitiesMixed}.
    Finally, intersection cuts for non-polyhedral sets and certain classes of disjunctions are studied in \cite{ModaresiEtAl2016_IntersectionCutsNonlinear}.

    Nevertheless, for the most part, these works remain theoretical contributions.
    Indeed, with the exception of \cite{Cezik2005,AtamturkNarayanan2010_ConicMixedInteger,Modaresi2016_ValidInequalitiesReformulation}, no computational results were reported for any of these techniques and, to the best of our knowledge, none has been implemented in optimization solvers.
    In fact, neither Mosek nor Gurobi\footnote{Personal communication with Gurobi and Mosek developers.} generate cuts from non-linear information.

\subsection{Contribution and outline}
    
In this paper, we study linear disjunctive cutting planes for \MICONIC\ problems.
Our objective is to derive practical and numerically robust tools for the separation of those cuts, and we show that conic formulations allow us to achieve it.
Specifically, we do so by extending Balas' CGLP into a Cut-Generating Conic Program (CGCP) (see also \cite{Cezik2005}).
Our contributions are:
\begin{enumerate}
    \item We study the role of the normalization condition in the CGCP, and propose conic normalizations that guarantee strong duality.
    In doing so, we answer some concerns that were raised in previous works.
    Namely,
    \begin{itemize}
        \item With respect to \cite{Bonami2011,Cezik2005,KronqvistMisener2020_DisjunctiveCutStrengthening}, we can select the right normalization to overcome issues associated with potential lack of constraint qualification.
        \item With respect to \cite{Kilinc2017,KronqvistMisener2020_DisjunctiveCutStrengthening}, since we use conic formulations, we do not need \emph{(i)} to pay attention at avoiding generating linearization cuts at points outside the domain where the non-linear functions are known to be convex, \emph{(ii)} to deal with non-differentiable functions, and \emph{(iii)} boundedness assumptions on the value of the constraints and their gradients.
    \end{itemize}
    \item We draw attention to some limitations of separating conic-infeasible points in an outer-approximation context, and propose algorithmic strategies to alleviate them.
    \item We introduce conic extensions of the lifting procedure for disjunctive cuts, and of monoidal strengthening for split cuts. 
    \item We provide computational results on the effectiveness of the proposed approach, thereby showing the benefits of the conic representation, and compare the practical effectiveness of several normalization conditions.
    \item We make our implementation available\footnote{Our code is released at \url{https://github.com/mtanneau/CLaP}} under an open-source license.
\end{enumerate}

The remainder of the paper is structured as follows. 
In Section \ref{sec:background}, we introduce some required notation and background material on conic optimization, and state a number of theoretical results on the characterization of valid inequalities for conic and disjunctive conic sets.
Section \ref{sec:separation} formalizes the CGCP and its dual, and the theoretical properties of several normalization conditions for the CGCP are discussed in Section \ref{sec:normalization}.
The separation of conic-infeasible points is further investigated in Section \ref{sec:infeas}, while classical lifting and strengthening techniques from MILP are extended to the conic setting in Section \ref{sec:LiftStrengthen}.
In Section \ref{sec:results}, we analyze the practical behavior of different normalizations, and show that our approach is competitive with CPLEX internal lift-and-project cuts.
Some concluding remarks are presented in Section \ref{sec:conclusion}.

\section{Background}
\label{sec:background}

    In this section, we introduce some notations, and recall a number of results that are needed for our approach.
    We refer to \cite{Rockafellar1970} for a thorough overview of convex analysis, and to \cite{Ben-Tal2001} and \cite{Balas1979} for results on conic optimization and disjunctive programming, respectively.

    For $\X \subseteq \R^{n}$, we denote by $\intr(\X)$, $\partial(\X)$, $\cl (\X)$, and $\conv(\X)$ the \emph{interior}, \emph{boundary}, \emph{closure}, and \emph{convex hull} of $\X$, respectively.
    The Minkowski sum of $\X, \Y \subseteq \R^{n}$ is defined by
    \begin{align*}
        \X + \Y = \left\{ x + y \ \middle| \ x \in \X, y \in \Y \right\}.
    \end{align*}
    
    If $\| \cdot \|$ is a norm on $\R^{n}$, its \emph{dual norm} $\| \cdot \|_{*}$ is defined by
    \begin{align}
        \forall y \in \R^{n}, \ \| y \|_{*} = \sup \left\{ y^{T}x \ \middle| \ \|x \| \leq 1 \right\}.
    \end{align}
    In all that follows, $\|\cdot\|_{2}$ denotes the Euclidean norm on $\R^{n}$.
    Finally, we denote by $e$ a vector of all ones, and by $e_{j}$ a vector whose $j^{th}$ coordinate is $1$ and all others are $0$; the dimension of $e$ and $e_{j}$ is always obvious from context.
    
\subsection{Cones and conic duality}
    
    The set $\K \subseteq \R^{n}$ is a \emph{cone} if $\forall (x, \lambda) \in \K \times \R_{+}, \ \lambda x \in \K$, and it is \emph{irreducible} if it cannot be written as a cartesian product of irreducible cones.
    The \emph{dual cone} of $\K \subseteq \R^{n}$ is
    \begin{align}
        \label{eq:def:dual_cone}
        \K^{*} &= \left\{ 
            u \in \R^{n} 
        \ \middle| \
            u^{T}x \geq 0, \forall x \in \K  
        \right\},
    \end{align}
    and $\K$ is \emph{self-dual} if $\K = \K^{*}$.
    A cone $\K \subseteq \R^{n}$ is \emph{pointed} if $\K \cap (-\K) = \{0\}$, i.e., if it does not contain a line that passes through the origin.
    \emph{Proper} cones are closed, convex, pointed cones with non-empty interior.
    If $\K$ is a proper cone, then $\K^{*}$ is also a proper cone, and any $\rho \in \intr \K$ induces a norm on $\K^{*}$, denoted by $\knorm{\rho}{\cdot}$ and defined by
    \begin{align}
    \label{eq:conicNorm}
        \knorm{\rho}{u} &= \rho^{T}u, \ \ \ \forall u \in \K^{*}.
    \end{align}
    
    Examples of proper cones include the non-negative orthant
    \begin{align*}
        \R_{+}^{n} = \left\{ x \in \R^{n} \ \middle| \ x \geq 0 \right\},
    \end{align*}
    the second-order cone (SOC)
    \begin{align*}
        \mathcal{L}_{n} = \left\{ x \in \R^{n} \ \middle| \ x_{1} \geq \sqrt{x_{2}^{2} + ... + x_{n}^{2}} \right\},
    \end{align*}
    the positive semi-definite (PSD) cone
    \begin{align*}
        \mathbb{S}_{+}^{n} = \left\{ X \in \R^{n \times n} \ \middle| \ X = X^{T}, \lambda_{min}(X) \geq 0 \right\},
    \end{align*}
    where $\lambda_{min}(X)$ is the smallest eigenvalue of $X$, and the exponential cone
    \begin{align*}
        \mathcal{E} = \cl \left\{ (x, y, z) \in \R^{3} \ \middle| \ x \exp (\nicefrac{x}{y}) \leq z, y > 0 \right\}.
    \end{align*}
    The non-negative orthant, SOC and SDP cone are also self-dual, while the exponential cone is not.
    
    A proper cone $\K$ induces a partial (resp. strict partial) ordering on $\R^{n}$, denoted $\Kgeq{\K}$ (resp. $\Kgt{\K}$) and defined by
    \begin{align}
        \label{eq:def:conic_inequality}
        \forall (x, y) \in \R^{n} \times \R^{n}, \ x \Kgeq{\K} y & \Leftrightarrow x - y \in \K,\\
        \label{eq:def:conic_inequality_strict}
        \forall (x, y) \in \R^{n} \times \R^{n}, \ x \Kgt{\K} y  & \Leftrightarrow x - y \in \intr (\K).
    \end{align}
    In all that follows, we refer to $A x \Kgeq{\K} b$ (resp. $Ax \Kgt{\K} b$) as a \emph{conic} (resp. \emph{strict conic}) inequality.
    Consider the system
    \begin{align}
        \label{eq:conic_system}
        A x \Kgeq{\K} b,
    \end{align}
    where $A \in \R^{m \times n}$ and $\K = \K_{1} \times ... \times \K_{N}$ with $\K_{i} \subset \R^{m_{i}}$; correspondingly, for $y \in \R^{m}$, we write $y = (y_{1}, ..., y_{N})$.
    We follow the terminology of \cite{Friberg2016_FacialReductionHeuristics}, and say that system \eqref{eq:conic_system} is \emph{feasible} if there exists $x \in \R^{n}$ such that $A x - b \in \K$, and \emph{strongly feasible} if there exists $x \in \R^{n}$ such that $A x - b \in \K$ and $(Ax - b)_{i} \in \intr(\K_{i})$ for all non-polyhedral cones $\K_{i}$, i.e., such that all non-polyhedral conic inequalities are strictly satisfied.
    Similarly, system \eqref{eq:conic_system} is \emph{infeasible} if it does not admit any feasible solution,
    and \emph{strongly infeasible} if, in addition, there exists $y \in \K^{*}$ such that $A^{T}y=0$ and $b^{T}y > 0$.
    Furthermore, we say that system \eqref{eq:conic_system} is
    \emph{weakly feasible} if it is feasible but not strongly feasible,
    and \emph{weakly infeasible} if it is infeasible but not strongly infeasible.
    Finally, a system is \emph{well-posed} if it is either strongly feasible or strongly infeasible, and \emph{ill-posed} otherwise.
    
    Let us emphasize that well-posedness is an algebraic property, i.e., it is not associated to a geometric set but to its algebraic representation through conic inequalities.
    For instance, for $n \geq 3$, both $0 \Kgeq{\R^{n}_{+}} x \Kgeq{\R^{n}_{+}} 0$ and $0 \Kgeq{\mathcal{L}_{n}} x \Kgeq{\mathcal{L}_{n}} 0$ describe the same set $\{ 0\}$, however, the former is well-posed and the latter is not.
    Nevertheless, for brevity, we will refer to the well-posedness of a set $\X$, only if there no ambiguity in its description with conic inequalities.
    
    A conic optimization problem writes, in standard form,
    \begin{subequations}
    \label{eq:ConicP}
    \begin{align}
        (P) \ \ \ \min_{x} \ \ \ & c^{T}x\\
        s.t. \ \ \ 
        & Ax = b,\\
        & x \in \K,
    \end{align}
    \end{subequations}
    where $A \in \R^{m \times n}$, $b \in \R^{m}$, and $\K$ is a cone.
    The strong/weak (in)feasibility and well-posedness of (P) refers to that of the system $(Ax = b, x \in \K)$.
    The optimal value of (P) is $\opt(P) = \inf \left\{ c^{T}x \ \middle| \ Ax =b, x \in \K \right\}$, and we say that (P) is \emph{bounded} if  $\opt(P) \in \R$ and \emph{solvable} if, in addition, there exists a feasible solution $x^{*}$ such that $c^{T}x^{*} = \opt(P)$.
    The dual of (P) is
    \begin{subequations}
    \label{eq:ConicD}
    \begin{align}
        (D) \ \ \ \max_{y, s} \ \ \ & b^{T}y\\
        s.t. \ \ \ 
        & A^{T}y + s = c,\\
        & s \in \K^{*},
    \end{align}
    \end{subequations}
    and $\opt(D) = \sup \left\{ b^{T}y \ \middle| \ A^{T}y + s = c, s \in \K^{*} \right\}$.
    In particular, (D) is also a conic optimization problem.
    
    \begin{theorem}[Conic duality theorem]
    \label{thm:conic duality}
        \begin{enumerate}
            \item \, [Weak duality] $\opt(D) \leq \opt(P)$.
            \item \, [Strong duality] If (P) (resp. (D)) is strongly feasible and bounded, then (D) (resp. (P)) is solvable and $\opt(P) = \opt(D)$.\\
            If both (P) and (D) are strongly feasible, then both are solvable with same optimal value.
        \end{enumerate}
    \end{theorem}
    \begin{proof}
        See Theorem 1.4.4 in \cite{Ben-Tal2001}.
        \qed
    \end{proof}
    
    Conic duality extends the classical duality for linear programming, albeit with a number of edge cases that lead to practical difficulties.
    For instance, there may exist a positive duality gap even though both (P) and (D) are solvable, as illustrated in Example \ref{ex:DualityGap}.
    
    \begin{example}[Example 8.6, \cite{Mosek2020_ModelingCookbook}]
    \label{ex:DualityGap}
        Consider the primal-dual pair
        \begin{align*}
            (P) \ \ \ \min_{x_{1}, x_{2}, x_{3}} \ \ \ 
            & x_{3}
                & (D) \ \ \ \max_{y_{1}, y_{2}} \ \ \ 
                & -y_{2} \\
            s.t. \ \ \ 
            & x_{2} \geq x_{1},
                & s.t. \ \ \ 
                & (y_{1}, -y_{1}, 1-y_{2}) \in \mathcal{L}_{3},\\
            & x_{3} \geq -1,
                && y_{1}, y_{2} \geq 0.\\
            & (x_{1}, x_{2}, x_{3}) \in \mathcal{L}_{3},
        \end{align*}
        Primal-feasible solutions are of the form $(x_{1}, x_{1}, 0)$, while dual-feasible solutions are of the form $(y_{1}, 1)$.
        Thus, $\opt(P) = 0$ and $\opt(D) = -1 < \opt(P)$.
    \end{example}
    
\subsection{Valid inequalities}
    
    For $(\alpha, \beta) \in \R^{n} \times \R$, we say that $\alpha^{T}x \geq \beta$ is a \emph{valid inequality} for $\X \subseteq \R^{n}$ if
    \begin{align*}
        \X \subseteq \left\{
            x \in \R^{n} 
        \ \middle| \ 
            \alpha^{T}x \geq \beta
        \right\},
    \end{align*}
    and a \emph{supporting hyperplane} if, in addition, $\exists \, \tilde{x} \in \cl \conv \X: \alpha^{T} \tilde{x} = \beta$.
    The set of valid inequalities for $\X$ is denoted by $\polar{\X}$, i.e.,
    \begin{align}
        \label{eq:def:polar}
        \polar{\X} =
        \left\{
            (\alpha, \beta) \in \R^{n} {\times} \R
        \ \middle| \ 
            \forall x \in \X, \ \alpha^{T}x \geq \beta
        \right\}.
    \end{align}
    Note that $\polar{\X}$ is a closed, convex set, and that
    \begin{align}
        \left\{ x \in \R^{n}
        \ \middle|\ 
            \forall (\alpha, \beta) \in \polar{\X}, \ \alpha^{T}x \geq \beta
        \right\}
        = \cl \conv(\X).
    \end{align}
    Furthermore, for $\X, \Y \subseteq \R^{n}$, we have
    \begin{align*}
        \X \subseteq \Y & \Rightarrow \polar{\Y} \subseteq \polar{\X},\\
        \polar{(\X \cup \Y)} &= \polar{\X} \cap \polar{\Y}.
    \end{align*}
    
    We now focus on the case where $\X$ is described by conic inequalities, and seek an algebraic description of $\polar{\X}$ using a finite number of conic inequalities.
    Note that, although $\polar \X$ is described by an infinite number of linear inequalities, as per Equation \eqref{eq:def:polar}, this semi-infinite representation is not computationally tractable.
    
    \begin{theorem}[Conic theorem on alternatives]
    \label{thm:alternatives}
        Consider the conic system
        \begin{align}
        \label{eq:thm:alternatives:conic system}
            Ax \Kgeq{\K} b,
        \end{align}
        where $A \in \R^{m \times n}$ has full column rank and $\K$ is a proper cone.
        \begin{enumerate}
            \item \label{thm:alternatives:farkas}
            If there exists $y \in \R^{m}$ such that
            \begin{align}
            \label{eq:thm:alternatives:dual}
                A^{T}y = 0, b^{T}y > 0, y \in \K^{*},
            \end{align}
            then \eqref{eq:thm:alternatives:conic system} has no solution.
            
            \item \label{thm:alternatives:almost solvable}
            If \eqref{eq:thm:alternatives:dual} has no solution, then \eqref{eq:thm:alternatives:conic system} is almost solvable, i.e., for any $\epsilon > 0$, there exists $\tilde{b} \in \R^{m}$ such that $\| b - \tilde{b} \|_{2} < \epsilon$ and the system $Ax \Kgeq{\K} \tilde{b}$ is solvable.
            
            \item \label{thm:alternatives:alternative}
            \eqref{eq:thm:alternatives:dual} is solvable if and only if \eqref{eq:thm:alternatives:conic system} is not almost solvable.
        \end{enumerate}
    \end{theorem}
    \begin{proof}
        See Proposition 1.4.2 in \cite{Ben-Tal2001}.
        \qed
    \end{proof}
    
    \begin{theorem}[Valid inequalities]
    \label{thm:background:ValidInequalities}
        Let $\C = \{x \ | \ Ax \Kgeq{\K} b \}$ with $A$ of full column rank, and define
        \begin{align*}
            \F &= \left\{
                (\alpha, \beta) 
            \ \middle| \
                \exists u \in \K^{*}: (\alpha = A^{T}u, \beta \leq b^{T} u)
            \right\}.
        \end{align*}
        Then, $\cl \F \subseteq \polar{\C}$ and, in addition,
        \begin{enumerate}
            \item if $\C \neq \emptyset$, then $\cl \F = \polar{\C}$;
            \item if $\C$ is well-posed, then $\F = \polar{\C}$.
        \end{enumerate}
    \end{theorem}
    \begin{proof}
        The inclusion $\F \subseteq \polar{C}$ is immediate from the definition of $\K^{*}$, and it follows that $\cl(\F) \subseteq \cl(\polar{C}) = \polar{C}$.
        Case \textbf{2} is a direct consequence of conic strong duality.
        
        We now prove \textbf{1}. Assume $\C \neq \emptyset$, let $(\alpha, \beta) \in \polar{\C}$, and consider the systems
        \begin{align}
            \label{eq:thm:I}
            A^{T}u = \alpha,& \ b^{T}u \geq \beta, \ u \in \K^{*};\\
            \label{eq:thm:II}
            Ax \Kgeq{\K} tb,& \ \alpha^{T}x < t \beta, \ t \geq 0.
        \end{align}
        By Theorem \ref{thm:alternatives}, either \eqref{eq:thm:I} is almost solvable, or \eqref{eq:thm:II} is solvable.
        Let us prove that the latter does not hold.
        
        Let $(x, t)$ be a solution to \eqref{eq:thm:II}.
        On the one hand, if $t > 0$, letting $\bar{x} = t^{-1}x$, we have $A \bar{x} \Kgeq{\K} b$, i.e., $\bar{x} \in \C$, but $\alpha^{T} \bar{x} < \beta$, which contradicts $(\alpha, \beta) \in \polar{\C}$.
        On the other hand, if $t = 0$, then we have $Ax \Kgeq{\K} 0$ and $\alpha^{T}x < 0$.
        Thus, for $x_{0} \in \C$ and $\tau \geq 0$, we have
        \begin{align*}
            A (x_{0} + \tau x) \Kgeq{\K} b,
        \end{align*}
        i.e., $(x_{0} + \tau x) \in \C$.
        Furthermore, we have $\alpha^{T}(x_{0} + \tau x) < \beta$ for large enough $\tau$, which also contradicts $(\alpha, \beta) \in \polar{\C}$.
        Therefore, \eqref{eq:thm:II} is not solvable and \eqref{eq:thm:I} is almost solvable.
        
        Thus, for any $\epsilon > 0$, there exists $\alpha_{\epsilon}$, $\beta_{\epsilon}$ and $u_{\epsilon} \in \K^{*}$ such that
        and 
        \begin{align*}
            \| \alpha - \alpha_{\epsilon} \|_{2} \leq \epsilon,
            \ \| \beta - \beta_{\epsilon} \|_{2} \leq \epsilon,
        \end{align*}
        and 
        \begin{align*}
            \alpha_{\epsilon} = A^{T} u_{\epsilon},
            \ \beta_{\epsilon} \leq b^{T}u_{\epsilon},
        \end{align*}
        i.e., $(\alpha_{\epsilon}, \beta_{\epsilon}) \in \F$.
        Taking $\epsilon \rightarrow 0$, we obtain that $(\alpha, \beta) \in \cl \F$.
        \qed
    \end{proof}
    
    We will refer to the multiplier $u \in \K^{*}$ in Theorem \ref{thm:background:ValidInequalities} as a  (conic) \emph{Farkas multiplier}, and we say that $\alpha^{T}x \geq \beta$ is obtained by Farkas aggregation if $\alpha = A^{T}u$ and $\beta \leq b^{T}u$ for some $u \in \K^{*}$.
    
    Theorem \ref{thm:background:ValidInequalities} highlights a fundamental difference between the linear and non-linear settings.
    In the linear case, $\mathcal{F}$ is polyhedral, thus, it is always closed and, if $\C$ is non-empty, then all valid inequalities for $\C$ can be obtained by Farkas aggregation.
    In the conic setting, however, this property may no longer hold, i.e., there may exist valid inequalities that cannot be represented through Farkas aggregation.
    
    \begin{example}
    \label{ex:background:illPosed}
        
        Let $\C = \{ x \ | \ Ax \Kgeq{\K} 0 \}$ where
        \begin{align*}
            A = \begin{pmatrix} 1 & 0 \\ 0 & 1 \\ 1 & 0 \end{pmatrix}, \ 
            \K = \mathcal{L}_{3},
        \end{align*}
        i.e., $\C = \left\{ (x, y) \in \R^{2}
            \ \middle| \ 
                (x, y, x) \in \mathcal{L}_{3}
            \right\}$.
        While $A$ has full column rank and $\K$ is proper, the system $Ax \Kgeq{\K} 0$ is not well-posed.
        It then is easy to verify that $\C = \left\{(x, 0) \ \middle| \ x \geq 0 \right\}$, and that $y \geq 0$ is a valid inequality for $\C$.
        
        However, for any $u \in \K^{*} = \mathcal{L}_{3}$, we have
        \begin{align*}
            A^{T}u & = \begin{pmatrix} u_{1} + u_{3} \\ u_{2} \end{pmatrix},
        \end{align*}
        and $u_{1} + u_{3} > 0$ unless $u_{2} = 0$.
        Therefore, the system $A^{T}u = (0, 1), u \in \K^{*}$ has no solution, i.e., the valid inequality $y \geq 0$ cannot be obtained by Farkas aggregation.
        
        Nevertheless, for $t \geq 0$, let $u_{t} = (\sqrt{t^{2} + 1}, -t, 1)$, which  yields the valid inequality $(\sqrt{t^{2}+1} - t)x + y \geq 0$.
        Then, as $t \rightarrow + \infty$, the term $(\sqrt{t^{2}+1} - t)$ becomes negligible and the inequality becomes, in the limit, $y \geq 0$.
    
    \end{example}
    
\subsection{Disjunctive inequalities}
    
    We now consider disjunctive conic sets, i.e., sets of the form
    \begin{align}
        \D &= \left\{
            x \in \R^{n}
        \ \middle| \
            \bigvee_{h=1}^{H} D_{h} x \Kgeq{Q_{h}} d_{h}
        \right\}
        \\
        & 
        = \bigcup_{h=1}^{H} \left\{
            x \in \R^{n}
        \ \middle| \ 
            D_{h} x \Kgeq{Q_{h}} d_{h}
        \right\},
    \end{align}
    where $H \in \Z_{+}$ and, $\forall h$, $D_{h} \in \R^{m_{h} \times n}$ and $\Q_{h}$ is a proper cone.
    We refer to $\conv \D$ as the \emph{disjunctive hull} and,
    for $(\alpha, \beta) \in \polar \D$, we say that $\alpha^{T}x \geq \beta$ is a \emph{disjunctive inequality}.
    
    We focus on disjunctive conic sets and, again, we seek a tractable algebraic characterization of valid inequalities for such sets.
    We begin by stating an extension to the conic setting of Balas' representation of the convex hull of a union of polyhedra \cite{Balas1979}.
    
    \begin{theorem}[Characterization of the convex hull]
    \label{thm:ConvexHull}
        Let
        \begin{align*}
            \D &= \bigcup_{h=1}^{H}
            \underbrace{
            \left\{
                x \in \R^{n}
            \ \middle| \ 
                D_{h} x \Kgeq{\Q_{h}} d_{h}
            \right\}
            }_{\D_{h}}
           ,
        \end{align*}
        where, $\forall h$, $D_{h} \in \R^{m_{h} \times n}$ and $\Q_{h}$ is a proper cone, and let
        \begin{align*}
            \mathcal{S} &= \left\{
                x \in \R^{n}
            \ \middle| \ 
                \exists (y_{1}, ..., y_{H}, z_{1}, ..., z_{H}): 
                \begin{array}{ll}
                    \sum_{h} y_{h} = x,\\
                    D_{h} y_{h} \Kgeq{\Q_{h}} z_{h} d_{h}, & \forall h,\\
                    z_{h} \geq 0, & \forall h,\\
                    \sum_{h} z_{h} = 1, & \forall h,\\
                    y_{h} \in \R^{n}, & \forall h\\
                \end{array}
            \right\}.
        \end{align*}
        Then $\conv(\D) \subseteq \mathcal{S}$ and, in addition,
        \begin{enumerate}
            \item if $\forall h, \D_{h} \neq \emptyset$, then $\mathcal{S} \subseteq \cl \conv(\D)$;
            \item if $\forall h, \D_{h} = \X_{h} + W$, where $\X_{1}, ..., \X_{H}$ are non-empty, closed, bounded, convex sets and $W$ is a closed convex set, then
            \begin{align*}
                \conv(\D) = \mathcal{S} = \cl \conv(\D).
            \end{align*}
        \end{enumerate}
    \end{theorem}
    \begin{proof}
        See Proposition 2.3.5 in \cite{Ben-Tal2001}.
        \qed
    \end{proof}
    
    Next, building on Farkas multipliers and the result of Theorem \ref{thm:background:ValidInequalities}, we can extend Balas' characterization of valid disjunctive inequalities (Theorem 3.1 in \cite{Balas1979}) to the conic setting.
    
    \begin{theorem}[Disjunctive inequalities]
    \label{thm:disjunctiveInequalities}
        Let
        \begin{align*}
            \D &= \bigcup_{h=1}^{H}
            \underbrace{
            \left\{
                x \in \R^{n}
            \ \middle| \ 
                D_{h} x \Kgeq{\Q_{h}} d_{h}
            \right\}
            }_{\D_{h}}
           ,
        \end{align*}
        where, $\forall h$, $D_{h} \in \R^{m_{h} \times n}$ and $\Q_{h}$ is a proper cone, and
        \begin{align*}
            \F = \bigcap_{h=1}^{H}
            \underbrace{
            \left\{
                (\alpha, \beta) \in \R^{n} \times \R
            \ \middle| \ 
                \exists u_{h} \in \Q_{h}^{*}: (\alpha = A^{T}u_{h}, \beta \leq b^{T} u_{h})
            \right\}
            }_{\F_{h}}
            .
        \end{align*}
        Then, $\F \subset \polar{\D}$ and, in addition,
        \begin{enumerate}
            \item if $\forall h, \D_{h} \neq \emptyset$, then $\polar{\D} = \bigcap_{h} \cl \F_{h}$;
            \item if $\forall h, \D_{h}$ is well-posed and $D_{h}$ has full column rank, then $\F = \polar{\D}$.
        \end{enumerate}
    \end{theorem}
    \begin{proof}
        Immediate from Theorem \ref{thm:background:ValidInequalities} and the fact that $\polar{\D} = \bigcap_{h} \polar{D_{h}}$.
        \qed
    \end{proof}
    
    \begin{example}
    \label{ex:background:ill-posed disjunction}
        Let
        \begin{align*}
            \D = 
            \left\{ (x, y) \in \R^{2}
            \ \middle| \ 
                (x, y, x) \in \mathcal{L}_{3}
            \right\}
            \cup
            \left\{
                (x, y) \in \R^{2}
            \ \middle| \ 
                (-x, y, x) \in \mathcal{L}_{3}
            \right\}
            .
        \end{align*}
        Thus, $\D = \{ (x, 0) | x \in \R \}$, and valid inequalities for $\D$ are of the form $\pm y \geq \pm \beta$ for $\beta \leq 0$.
        Building on Example \ref{ex:background:illPosed}, it follows that $\F = \{(0, 0)\} \times \R_{-}$.
        Therefore, only the trivial inequality $0 \geq -1$ can be represented using finite Farkas multipliers for each disjunctive term.
    \end{example}

\section{Cut separation}
\label{sec:separation}

We consider an \MICONIC\ problem of the form
\begin{subequations}
\label{eq:MICP}
\begin{align}
	(MICP) \ \ \ 
	\label{eq:MICP:objective}
	\min_{x} \ \ \ & c^{T}x \\
	s.t. \ \ \
	\label{eq:MICP:eq}
	& A x = b\\
	\label{eq:MICP:conic}
	& x \in \K,\\
	\label{eq:MICP:integer}
	& x \in \Z^{p} \times \R^{n-p},
\end{align}
\end{subequations}
where $A \in \R^{m \times n}$, $p \leq n$, and $\K$ is a proper cone.
The continuous relaxation of $(MICP)$, denoted by $(CP)$, is given by \eqref{eq:MICP:objective}-\eqref{eq:MICP:conic}.
The feasible sets of $(MICP)$ and $(CP)$ are denoted by $\X$ and by $\C$, respectively.
    
Let $\bar{x} \in \R^{n}$ be a point to separate.
Since $\bar{x}$ is typically obtained from solving a relaxation of $(MICP)$, we will assume that $A \bar{x} = b$, i.e., all linear equality constraints are satisfied.
In particular, we will not assume that $\bar{x}$ is conic-feasible, i.e., we may have $\bar{x} \notin \K$, for instance when an outer-approximation algorithm is used.

Consider the disjunctive set
\begin{align}
    \label{eq:split_disjunction}
    \mathcal{D} = \bigcup_{h \in H} \D_{h}
    = \bigcup_{h \in H}
    \left\{
        x
        \ \middle| \ 
        \begin{array}{l}
            A x = b, x \in \K\\
            D_{h} x \Kgeq{\Q_{h}} d_{h}
        \end{array}
    \right\},
\end{align}
where each $\Q_{h}$ is a proper cone, and $\D \supseteq \X$.
Valid inequalities for $\D$ are referred to as \emph{disjunctive inequalities} or, equivalently, as \emph{disjunctive cuts}.
For $(\alpha, \beta) \in \polar{\D}$, the inequality $\alpha^{T}x \geq \beta$ is \emph{trivial} if $(\alpha, \beta) \in \polar{\C}$, and \emph{non-trivial} otherwise.
Following the terminology of \cite{coey2018outer}, \emph{$\K^{*}$ cuts} are trivial inequalities of the form $u^{T}x \geq 0$ for $u \in \K^{*}$.
Finally, a cut is \emph{violated} if $\alpha^{T}\bar{x} < \beta$.

It is always possible, e.g., through the use of facial reduction techniques \cite{Friberg2016_FacialReductionHeuristics}, to describe each $\D_{h}$ by a well-posed system of conic inequalities.
In addition, one may assume, after manual inspection, that $\forall h, \D_{h} \neq \emptyset$.
However, in a cutting-plane context, systematically performing such reductions and verifications can quickly become intractable.
Therefore, unless stated otherwise, we make no assumption regarding the feasibility nor well-posedness of individual disjunctive terms; we will show in Section \ref{sec:normalization} how to address such shortcomings in a systematic way.

\subsection{Separation problem}
\label{sec:separation:CGCP}

    Consider the \emph{cut-generating conic problem} (CGCP)
    \begin{subequations}
    \label{eq:CGCP}
    \begin{align}
        \label{eq:CGCP:obj}
        (CGCP) \ \ \ \min_{\alpha, \beta, u, \lambda, v} \ \ \ & \alpha^{T} \bar{x} - \beta\\
        s.t. \ \ \ 
        & \label{eq:CGCP:f1} \alpha = A^{T}u_{h} + \lambda_{h} + D_{h}^{T} v_{h}, && \forall h,\\
        & \label{eq:CGCP:f2} \beta  \leq b^{T}u_{h} + d_{h}^{T}v_{h}, && \forall h,\\
        & \label{eq:CGCP:domain} (u_{h}, \lambda_{h}, v_{h}) \in \R^{m} \times \K^{*} \times \Q_{h}^{*}, && \forall h,
    \end{align}
    \end{subequations}
    which naturally extends Balas' CGLP to the conic setting, see also \cite{Cezik2005}.
    In particular, it is a conic programming problem, which can be solved by, e.g., an interior-point algorithm.
    
    First, it follows from Theorem \ref{thm:disjunctiveInequalities} that, if $(\alpha, \beta, u, \lambda, v)$ is feasible for \eqref{eq:CGCP}, then $\alpha^{T}x \geq \beta$ is a disjunctive inequality.
    Under the stronger assumption of Case 2. in Theorem \ref{thm:disjunctiveInequalities}, every disjunctive inequality corresponds to a feasible solution of the CGCP.
    In the absence of such assumptions, however, the exact characterization of $\polar{\D}$ stated in Theorem \ref{thm:disjunctiveInequalities} may not hold.
    For instance, there may exist disjunctive inequalities that do not correspond to any feasible solution of the CGCP; as illustrated by Example \ref{ex:background:ill-posed disjunction}, it may even be that all feasible solutions of the CGCP correspond to trivial inequalities of the form $0 \geq \beta$ for some $\beta \leq 0$.
    
    Second, the feasible set of the CGCP is an unbounded cone, which contains the origin.
    Thus, the CGCP is either unbounded or bounded with objective value zero.
    In the former case, any unbounded ray yields a violated cut, while in the latter, no violated cut is obtained.
    Note that unbounded problems can lead to numerical issues for some interior-point algorithms.
    Therefore, it is common practice to add a normalization condition to the CGCP, whose role is further investigated in Section \ref{sec:normalization}.
    
    Third, the CGCP is strongly feasible.
    Indeed, let $\bar{u}_{h} = 0, \bar{v}_{h} \in \intr \Q_{h}^{*}, \forall h$.
    Then, since $\K^{*}$ has non-empty interior, there exists $\bar{\alpha}$ such that
        \begin{align*}
            \bar{\alpha} \succ_{\K^{*}} D_{h}^{T}\bar{v}_{h}, \forall h.
        \end{align*}
    Finally, letting $\bar{\lambda}_{h} = \bar{\alpha} - D_{h}^{T} \bar{v}_{h} \in \intr{\K^{*}}$ and $\bar{\beta} \leq d_{h}^{T} \bar{v}_{h}, \forall h$, it follows that $(\bar{\alpha}, \bar{\beta}, \bar{u}, \bar{\lambda}, \bar{x})$ is strongly feasible for the CGCP.
    
    Fourth, let $(\alpha, \beta, u, \lambda, v)$ be a feasible solution of the CGCP, and let
    \begin{align}
        (\tilde{\alpha}, \tilde{\beta}, \tilde{u}, \tilde{\lambda}, \tilde{v})
        &= (\alpha - A^{T}u_{0}, \beta - b^{T}u_{0}, u - u_{0}, \lambda, v),
    \end{align}
    where $u_{0} \in \R^{m}$.
    It is immediate to see that $(\tilde{\alpha}, \tilde{\beta}, \tilde{u}, \tilde{\lambda}, \tilde{v})$ is also feasible for the CGCP, with identical objective value since $A\bar{x} = b$.
    Therefore, without loss of generality, one of the $u_{h}$ can be arbitrarily set to zero in the formulation of the CGCP, thereby reducing its size.
    Furthermore, since $\C \subseteq \left\{ x \ \middle| \ Ax = b \right\}$, it follows that
    \begin{align}
    \label{eq:separation:EquivalentCuts}
        \C \cap \left\{ x \ \middle| \ \alpha^{T}x \geq \beta \right\}
        = \C \cap \left\{ x \ \middle| \ \tilde{\alpha}^{T}x \geq \tilde{\beta} \right\},
    \end{align}
    i.e., the two inequalities are equivalent in the sense that both cut off the same portion of the continuous relaxation.
    
\subsection{Membership problem}
\label{sec:separation:MCP}

    The dual problem of the CGCP is the \emph{membership conic problem} (MCP)
    \begin{subequations}
    \label{eq:MCP}
    \begin{align}
        \label{eq:MCP:obj}
        (MCP) \ \ \ \max_{y, z} \ \ \ & 0\\
        s.t. \ \ \
        \label{eq:MCP:friends}  & \sum_{h} y_{h} = \bar{x},\\
        \label{eq:MCP:convex}   & \sum_{h} z_{h} = 1,\\
        \label{eq:MCP:equality}  & A y_{h} = z_{h} b, & \forall h,\\
        \label{eq:MCP:disj}     & D_{h} y_{h} \Kgeq{\Q_{h}} z_{h} d_{h}, & \forall h,\\
        \label{eq:MCP:domain}    & (y_{h}, z_{h}) \in \K \times \R_{+}, & \forall h,
    \end{align}
    \end{subequations}
    which extends Bonami's membership LP \cite{Bonami2012a} to the conic setting.
    
    A geometrical interpretation of the MCP is provided by Theorem \ref{thm:ConvexHull}.
    If all disjunctive terms are non-empty and have identical recession cones, then \eqref{eq:MCP} is feasible if and only if $\bar{x} \in \conv(\D)$.
    In the general case, however, the exact characterization of $\conv \D$ given by Theorem \ref{thm:ConvexHull} may no longer hold, and we can only state that, if $\bar{x} \in \conv \D$, then the MCP is feasible.
    
    By weak duality, if the MCP is feasible, then the objective value of the CGCP is bounded below.
    If, in addition, the MCP is strongly feasible, then both the MCP and the CGCP are solvable with identical objective values.

\section{The roles of normalization}
\label{sec:normalization}

\newcommand{\drawContRelax}[1]{
    \fill [line width=0.0pt, fill=gray!20!white] (0.0, 0.0) circle (#1);
}

\newcommand{\drawSplitHull}[3]{
    \fill[orange!50!white] ({acos(#1/#3)}:#3) arc ({acos(#1/#3)}:{180+acos(#1/#3)}:#3);  
    \fill[orange!50!white] (-{acos(#2/#3)}:#3) arc (-{acos(#2/#3)}:{acos(#2/#3)}:#3);  
    \fill[orange!50!white, opacity=0.50] (-{acos(#2/#3)}:#3) -- ({acos(#2/#3)}:#3) -- ({acos(#1/#3)}:#3) -- ({180+acos(#1/#3)}:#3) -- (-{acos(#2/#3)}:#3);
    
}

\newcommand{\drawXbar}[2]{
    \draw [fill=black] (#1, #2) circle (1pt) node (xbar) {};
    \node[above] at ({#1+0.02}, {#2+0.02}) {$\bar{x}$};
}

This section focuses on the roles of the normalization condition in the CGCP.
On the one hand, through the lens of conic duality for the CGCP-MCP pair, we investigate the impact of the normalization on the solvability of CGCP.
On the other hand, by characterizing optimal solutions of the normalized CGCP, we assess the theoretical properties of the corresponding cuts.

The following normalization conditions are considered:
\begin{enumerate}[label=(\roman*)]
    \item the $\alpha$ normalization: $\| \alpha \|_{*} \leq 1$,
    \item the \emph{polar} normalization: $\gamma^{T}\alpha \leq 1$,
    \item the \emph{standard} normalization: $\sum_{h} \knorm{\rho}{\lambda_{h}} + \knorm{\sigma_{h}}{v_{h}} \leq 1$,
    \item the \emph{trivial} normalization: $\sum_{h} \knorm{\sigma_{h}}{v_{h}} \leq 1$,
    \item the \emph{uniform} normalization: $\sum_{h} \knorm{\rho}{\lambda_{h}} \leq 1$,
\end{enumerate}
where $\gamma \in \R^{n}$, $\rho \in \intr \K$ and $\sigma_{h} \in \intr \Q_{h}$.

Throughout this section, the strengths and shortcomings of each normalization are illustrated in a simple setting, described in Example \ref{ex:normalization} below.

\begin{example}
\label{ex:normalization}

    For $R > 0$, consider the MICP
    \begin{align}
        \min_{x} \ \ \ & -x_{1} - x_{2}\\
        s.t. \ \ \ &
        x_{0} = R,\\
        & x \in \mathcal{L}_{3},\\
        & x_{1}, x_{2} \in \Z,
    \end{align}
    and the split disjunction $(x_{1} \leq 0) \vee (x_{1} \geq 1)$.
    Thus, we have
    \begin{align*}
        \D_{1} &= \left\{ x \in \R^{3} \ \middle| \ x \in \mathcal{L}_{3}, x_{0} = R, x_{1} \leq 0 \right\},\\
        \D_{2} &= \left\{ x \in \R^{3} \ \middle| \ x \in \mathcal{L}_{3}, x_{0} = R, x_{1} \geq 1 \right\}.
    \end{align*}
    We consider the following three cases:
    \begin{enumerate}[label=(\alph*)]
        \item \label{ex:enum:well-posed} $R > 1$: $\D_{1}$ and $\D_{2}$ are both strongly feasible;
        \item \label{ex:enum:ill-posed} $R = 1$: $\D_{1}$ is strongly feasible and $\D_{2}$ is weakly feasible;
        \item \label{ex:enum:infeasible} $R < 1$: $\D_{1}$ is strongly feasible and $\D_{2} = \emptyset$.
    \end{enumerate}
    Each of these settings is illustrated in Figure \ref{fig:ex:normalization}.
    Finally, unless specified otherwise, $\bar{x}$ is the solution of the continuous relaxation, i.e., $\bar{x} = (R, \frac{R}{\sqrt{2}}, \frac{R}{\sqrt{2}})$.
    All CGCPs are solved as conic problems using Mosek 9.2 with default parameters.
    
    \begin{figure}
        \centering
        \subfloat[$R = 1.1$]{
        \label{fig:ex:norm:well-posed}
        \begin{tikzpicture}[scale=2]
            \clip(-0.5, -0.25) rectangle (1.25, 1.30);
            
            \drawContRelax{1.1}             
            \drawSplitHull{0.0}{1.0}{1.1}   
            \draw [line width=1pt] (0.0, -2.0) -- (0.0, 2.0);  
            \draw [line width=1pt] (1.0, -2.0) -- (1.0, 2.0);  
            
            \draw [line width=0.5pt, ->] (0.0, 1.20) -- (-0.20, 1.20);
            \draw [line width=0.5pt, ->] (1.0, 1.20) -- ( 1.20, 1.20);
            
            
            \drawXbar{0.778}{0.778}         
            
        \end{tikzpicture}
        }
        \hfill
        \subfloat[$R = 1.0$]{
        \label{fig:ex:norm:ill-posed}
        \begin{tikzpicture}[scale=2]
            \clip(-0.5, -0.25) rectangle (1.25, 1.30);
            
            \drawContRelax{1.0}             
            \drawSplitHull{0.0}{1.0}{1.0}   
            \draw [line width=1pt] (0.0, -2.0) -- (0.0, 2.0);  
            \draw [line width=1pt] (1.0, -2.0) -- (1.0, 2.0);  
            
            \draw [line width=0.5pt, ->] (0.0, 1.20) -- (-0.20, 1.20);
            \draw [line width=0.5pt, ->] (1.0, 1.20) -- ( 1.20, 1.20);
            
            \drawXbar{0.707}{0.707}         
            
        \end{tikzpicture}
        }
        \hfill
        \subfloat[$R = 0.9$]{
        \label{fig:ex:norm:infeasible}
        \begin{tikzpicture}[scale=2]
            \clip(-0.5, -0.25) rectangle (1.25, 1.30);
            
            \drawContRelax{0.9}             
            \fill[orange!50!white] (90:0.9) arc (90:270:0.9);  
            \draw [line width=1pt] (0.0, -2.0) -- (0.0, 2.0);  
            \draw [line width=1pt] (1.0, -2.0) -- (1.0, 2.0);  
            
            \draw [line width=0.5pt, ->] (0.0, 1.20) -- (-0.20, 1.20);
            \draw [line width=0.5pt, ->] (1.0, 1.20) -- ( 1.20, 1.20);

            \drawXbar{0.6364}{0.6364}         
            
        \end{tikzpicture}
        }
        
        \caption{The three settings from Example \ref{ex:normalization}, projected onto the $x_{0} = 1$ space. The domain of the continuous relaxation is in gray, the split hull in orange, and the split disjunction is indicated in black.}
        \label{fig:ex:normalization}
    \end{figure}
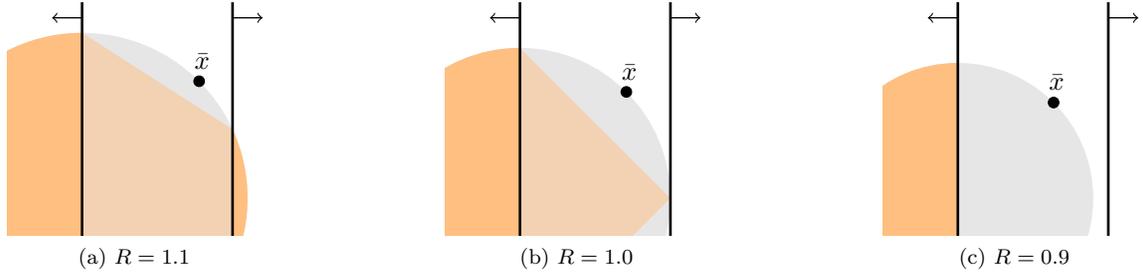
    
\end{example}

\subsection{$\alpha$ normalization}
\label{sec:normalization:alpha}

    A straightforward way of bounding the CGCP is to restrict the magnitude of $\alpha$.
    This approach was considered in previous work on \MICONV\ \cite{Ceria1999,Stubbs1999,Kilinc2017}, wherein authors considered restricting the $\ell_{1}$, $\ell_{\infty}$ or $\ell_{2}$ norm of $\alpha$.
    
    For a given norm $\|{\cdot}\|$, the CGCP then writes
    \begin{subequations}
    \label{eq:CGCP:alpha}
    \begin{align}
        \min_{\alpha, \beta, \lambda, u, v} \ \ \ & \alpha^{T} \bar{x} - \beta\\
        s.t. \ \ \ 
        & \label{eq:CGCP:alpha:f1} \alpha = A^{T}u_{h} + \lambda_{h} + D_{h}^{T} v_{h}, && \forall h,\\
        & \label{eq:CGCP:alpha:f2} \beta  \leq b^{T}u_{h} + d_{h}^{T}v_{h}, && \forall h,\\
        & \label{eq:CGCP:alpha:domain} (u_{h}, \lambda_{h}, v_{h}) \in \R^{m} \times \K^{*} \times \Q_{h}^{*}, && \forall h,\\
        & \label{eq:CGCP:alpha:normalization} \| \alpha \|_{*} \leq 1,
    \end{align}
    \end{subequations}
    and the corresponding MCP, up to a change of sign in the objective value, is
    \begin{subequations}
    \label{eq:MCP:alpha}
    \begin{align}
        \label{eq:MCP:alpha:obj}
        \min_{x, y, z} \ \ \ & \|\bar{x} - x \|\\
        s.t. \ \ \
        \label{eq:MCP:alpha:friends}  & \sum_{h} y_{h} = x,\\
        \label{eq:MCP:alpha:convex}   & \sum_{h} z_{h} = 1,\\
        \label{eq:MCP:alpha:equality}  & A y_{h} = z_{h} b, & \forall h,\\
        \label{eq:MCP:alpha:disj}     & D_{h} y_{h} \Kgeq{\Q_{h}} z_{h} d_{h}, & \forall h,\\
        \label{eq:MCP:alpha:domain}    & (y_{h}, z_{h}) \in \K \times \R_{+}, & \forall h.
    \end{align}
    \end{subequations}
    
    Geometrically, the CGCP \eqref{eq:CGCP:alpha} looks for a deepest cut, i.e., one that maximizes the distance from $\bar{x}$ to the hyperplane $\alpha^{T}x = \beta$, as measured by $\|{\cdot}\|$.
    Correspondingly, the MCP \eqref{eq:MCP:alpha} computes a projection of $\bar{x}$ onto the set defined by \eqref{eq:MCP:alpha:friends}-\eqref{eq:MCP:alpha:domain}, with respect to $\|{\cdot}\|$.
    It is trivially feasible if at least one of the disjunctive terms is non-empty, which is always the case if $\X \neq \emptyset$.
    This ensures that the CGCP \eqref{eq:CGCP:alpha} is never unbounded.
    If, in addition, each disjunctive term is strongly feasible, then the MCP \eqref{eq:MCP:alpha} is strongly feasible and both the MCP and the CGCP are solvable.
    
    Split cuts obtained with the $\alpha$ normalization in the context of Example \ref{ex:normalization} are illustrated in Figure \ref{fig:ex:alpha}; these results are obtained with the (self-dual) $\ell_{2}$ norm in \eqref{eq:CGCP:alpha:normalization}.
    Furthermore, some statistics regarding the resolution of the CGCP are reported in Table \ref{tab:CGCP:alpha}, namely: the number of interior-point iterations (Iter), and the magnitude of $\alpha, u, \lambda, v$ in the obtained CGCP solution.
    
    \begin{figure}
        \centering
        \subfloat[]{
        \label{fig:ex:alpha:well-posed}
        \begin{tikzpicture}[scale=2]
            \clip(-0.5, -0.25) rectangle (1.25, 1.30);
            
            \drawContRelax{1.1}             
            \drawSplitHull{0.0}{1.0}{1.1}   
            \drawXbar{0.778}{0.778}         
            
            \draw [line width=1pt, color=red] (-0.50, 1.42) -- (1.50, 0.14);
            
        \end{tikzpicture}
        }
        \hfill
        \subfloat[]{
        \label{fig:ex:alpha:ill-posed}
        \begin{tikzpicture}[scale=2]
            \clip(-0.5, -0.25) rectangle (1.25, 1.30);
            
            \drawContRelax{1.0}             
            \drawSplitHull{0.0}{1.0}{1.0}   
            \drawXbar{0.707}{0.707}         
            
            \draw [line width=1pt, color=red] (-0.50, 1.50) -- (1.50, -0.50);
            
        \end{tikzpicture}
        }
        \hfill
        \subfloat[]{
        \label{fig:ex:alpha:infeasible}
        \begin{tikzpicture}[scale=2]
            \clip(-0.5, -0.25) rectangle (1.25, 1.30);
            
            \drawContRelax{0.9}             
            \fill[orange!50!white] (90:0.9) arc (90:270:0.9);  
            
            \drawXbar{0.6364}{0.6364}
            \draw [line width=1pt, color=red] (0.0, -0.50) -- (0.0, 1.20);
            
            
            
        \end{tikzpicture}
        }
        
        \caption{Split cuts (in red) obtained with the $\alpha$ normalization.}
        \label{fig:ex:alpha}
    \end{figure}
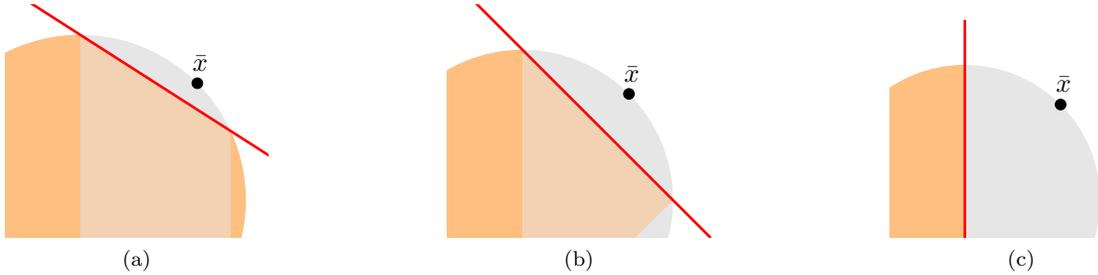
    
    \begin{table}
        \centering
        \caption{CGCP statistics for Example \ref{ex:normalization} and $\alpha$ normalization}
        \label{tab:CGCP:alpha}
        \begin{tabular}{ccrrrrrrrr}
        \toprule
            & Iter & $\| \alpha \|$ & $\| u_{1}\|$ & $\| u_{2}\|$ & $\|\lambda_{1}\|$ & $\|\lambda_{2}\|$ & $\| v_{1}\|$ & $\| v_{2}\|$ \\
        \midrule
            \ref{ex:enum:well-posed} & $ 8$ &     1.0 & $     0.8 $ & $     2.0 $ &     1.2 &     2.9 &     0.5 &     1.3\\
            \ref{ex:enum:ill-posed}  & $67^{*}$ &     1.0 & $     0.7 $ & $ 10958.8 $ &     1.0 & 15498.0 &     0.7 & 10958.1\\
            \ref{ex:enum:infeasible}  & $ 8$ &     1.0 & $     0.0 $ & $    13.7 $ &     0.0 &    19.3 &     1.0 &    12.6\\
        \bottomrule
        \end{tabular}\\
        $^{*}$: slow progress
    \end{table}
    
    In Example \ref{ex:normalization}\ref{ex:enum:well-posed}, the MCP is strongly feasible, no numerical trouble is encountered, and the obtained cut is a supporting hyperplane of the split hull.
    On the other hand, numerical issues are encountered in Example \ref{ex:normalization}\ref{ex:enum:ill-posed}.
    Indeed, as reported in Table \ref{tab:CGCP:alpha}, Mosek terminates due to slow progress after 67 iterations, albeit with a feasible solution, for which the magnitude of $u_{2}, \lambda_{2}, v_{2}$ is very large.
    The corresponding cut is displayed in Figure \ref{fig:ex:alpha:ill-posed}.
    Finally, in Example \ref{ex:normalization}\ref{ex:enum:infeasible}, although $\D_{2} = \emptyset$, the CGCP is solved without issue, thereby showing that it is not necessary for the MCP to be strongly feasible for the CGCP to be solvable.
    
    Example \ref{ex:normalization}\ref{ex:enum:ill-posed} illustrates the numerical challenges that may arise when the MCP is not strongly feasible.
    Indeed, in that case, the deepest cut writes $x_{1} + x_{2} \leq 1$, up to a positive scaling factor.
    However, although $x_{1} + x_{2} \leq 1$ is a valid inequality for $\D_{2}$, it is straightforward to see that it cannot be represented using finite Farkas multipliers: this corresponds to case \textbf{1} in Theorem \ref{thm:background:ValidInequalities}.
    Thus, the CGCP has no optimal solution, and there exists a (diverging) sequence of feasible solutions whose objective value becomes arbitrary close to $\opt(CGCP)$, corresponding to a sequence of valid inequalities that, in the limit, become equivalent to $x_{1} + x_{2} \leq 1$.
    Hence, the solver eventually runs into slow progress while the magnitude of $u_{2}, \lambda_{2}, v_{2}$ becomes large, as observed in Table \ref{tab:CGCP:alpha}.
    
    Interestingly, all cuts displayed in Figure \ref{fig:ex:alpha} are supporting hyperplanes of the split hull.
    A slightly more general result is stated in Theorem \ref{thm:alpha:tight}.
    
    \begin{theorem}
    \label{thm:alpha:tight}
        Assume that the CGCP \eqref{eq:CGCP:alpha} and the MCP \eqref{eq:MCP:alpha} are solvable, and let $(\alpha, \beta, u, \lambda, v)$ and $(x, y, z)$ be corresponding optimal solutions.
        Then, $$\alpha^{T}x = \beta.$$
    \end{theorem}
    \begin{proof}
        Let $\delta$ be the optimal value of the CGCP, i.e., $\delta = \alpha^{T}\bar{x} - \beta$.
        Since the CGCP is strongly feasible and bounded, by Theorem \ref{thm:conic duality}, strong duality holds.
        Thus, we have $\delta = -\|\bar{x} - x \|$.
        The case $\delta = 0$ is trivial, so we assume $\delta < 0$ and, thus, $x \neq \bar{x}$.
        
        Let $w = |\delta|^{-1}(x - \bar{x})$, i.e., $x = \bar{x} + |\delta| w$ and $\| w \| = 1$; in particular, we have $1 \geq \| \alpha \|_{*} \geq \alpha^{T}w$.
        It follows that
        \begin{align*}
            \alpha^{T}x - \beta
            &= \alpha^{T}\bar{x} + |\delta| \alpha^{T}w - \beta\\
            &= \delta + | \delta | \alpha^{T}w\\
            &= \delta(1 - \alpha^{T}w)\\
            & \leq 0.
        \end{align*}
        Thus, $\alpha^{T}x \leq \beta$.
        
        Next, we have
        \begin{align*}
            \alpha^{T}x
            &= \sum_{h} \alpha^{T}y_{h}\\
            &= \sum_{h} u_{h}^{T}A y_{h} + \lambda_{h}^{T}y_{h} + v_{h}^{T}D_{h}y_{h}\\
            & \geq \sum_{h} u_{h}^{T} (z_{h}b) + v_{h}^{T}(z_{h} d_{h})\\
            & \geq \sum_{h} z_{h} \beta\\
            &= \beta,
        \end{align*}
        which concludes the proof.
        \qed
    \end{proof}

    Therefore, if all disjunctive terms are non-empty, then, by Theorem \ref{thm:ConvexHull}, $x \in \cl \conv \D$, and the obtained inequality $\alpha^{T}x \geq \beta$ is indeed a supporting hyperplane of the disjunctive hull.
    

\subsection{Polar normalization}
\label{sec:normalization:polar}

    In the MILP setting, Balas and Perregard \cite{Perregaard2001,Balas2002} first suggest normalizing the CGLP with a single hyperplane of the form $\alpha^{T}\gamma = 1$.
    Doing so ensures that, if the CGLP is feasible and bounded, then there exists an optimal solution for which $(\alpha, \beta)$ is an extreme ray of $\polar{\left(\cl \conv \D \right)}$.
    This approach was then followed in \cite{CadouxLemarechal2013_ReflectionsGeneratingdisjunctive,Serra2020_ReformulatingDisjunctiveCut}, and more recently in \cite{ConfortiWolsey2019_FacetSeparationOne}.
    
    For $\gamma \in \R^{n}$, the CGCP writes
    \begin{subequations}
    \label{eq:CGCP:polar}
    \begin{align}
        \label{eq:CGCP:polar:obj}
        \min_{\alpha, \beta, u, v, \lambda} \ \ \ & \alpha^{T} \bar{x} - \beta\\
        s.t. \ \ \ 
        & \label{eq:CGCP:polar:f1} \alpha = A^{T}u_{h} + \lambda_{h} + D_{h}^{T} v_{h}, && \forall h,\\
        & \label{eq:CGCP:polar:f2} \beta  \leq b^{T}u_{h} + d_{h}^{T}v_{h}, && \forall h,\\
        & \label{eq:CGCP:polar:domain} (u_{h}, \lambda_{h}, v_{h}) \in \R^{m} \times \K^{*} \times \Q_{h}^{*}, && \forall h,\\
        & \label{eq:CGCP:polar:normalization} \alpha^{T}\gamma \leq 1,
    \end{align}
    \end{subequations}
    and the MCP is given by
    \begin{subequations}
    \label{eq:MCP:polar}
    \begin{align}
        \label{eq:MCP:polar:obj}
        \min_{y, z, \eta} \ \ \ & \eta\\
        s.t. \ \ \
        \label{eq:MCP:polar:friends}  & \sum_{h} y_{h} = \bar{x} + \eta \gamma,\\
        \label{eq:MCP:polar:convex}   & \sum_{h} z_{h} = 1,\\
        \label{eq:MCP:polar:constr1}  & A y_{h} = z_{h} b, & \forall h,\\
        \label{eq:MCP:polar:disj}     & D_{h} y_{h} \Kgeq{\Q_{h}} z_{h} d_{h}, & \forall h,\\
        \label{eq:MCP:polar:domain}    & (y_{h}, z_{h}) \in \K \times \R_{+}, & \forall h,\\
        \label{eq:MCP:polar:nonneg}   & \eta \geq 0.
    \end{align}
    \end{subequations}
    
    It follows from Theorem \ref{thm:ConvexHull} that, if there exists $\eta \geq 0$ such that $(\bar{x} + \eta \gamma) \in \conv \D$, then the MCP \eqref{eq:MCP:polar} is feasible.
    This is always the case if $\gamma = x^{*} - \bar{x}$, for some $x^{*} \in \conv \D$.
    If, in addition, each individual disjunction is strongly feasible and $x^{*}$ is obtained as a convex combination of strongly feasible points, then the MCP \eqref{eq:MCP:polar} is strongly feasible.
    
    Split cuts obtained for Example \ref{ex:normalization} with the polar normalization are illustrated in Figure \ref{fig:ex:polar}, and the corresponding CGCP statistics are reported in Table \ref{tab:CGCP:polar}.
    In each case, we set $\gamma = x^{*} - \bar{x}$, where $x^{*} = (R, 0 ,0)$.
    
    In all three cases, the obtained cut is identical to the one obtained with the $\alpha$ normalization, although this is not the case in general.
    Furthermore, as reported in Table \ref{tab:CGCP:polar}, numerical issues are also encountered for Example \ref{ex:normalization}\ref{ex:enum:ill-posed}, for the same reasons as for the $\alpha$ normalization: the CGCP is not solvable, and Mosek terminates with slow progress while the magnitude of $u_{2}, \lambda_{2}, v_{2}$ diverges.

    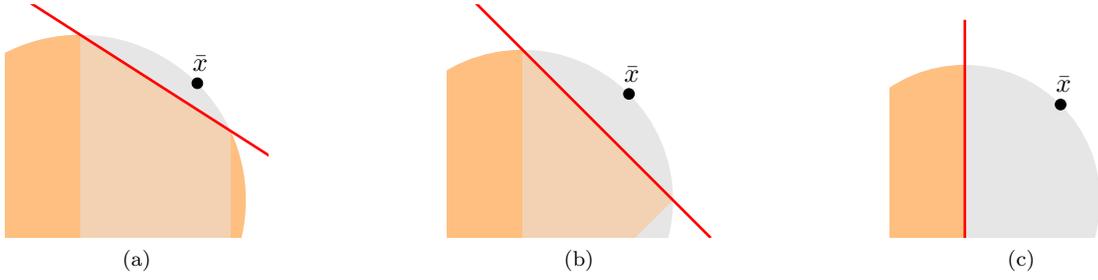
\begin{figure}
        \centering
        \subfloat[]{
        \label{fig:ex:polar:well-posed}
        \begin{tikzpicture}[scale=2]
            \clip(-0.5, -0.25) rectangle (1.25, 1.30);
            
            \drawContRelax{1.1}             
            \drawSplitHull{0.0}{1.0}{1.1}   
            \drawXbar{0.778}{0.778}         
            
            \draw [line width=1pt, color=red] (-0.50, 1.42) -- (1.50, 0.14);
            
        \end{tikzpicture}
        }
        \hfill
        \subfloat[]{
        \label{fig:ex:polar:ill-posed}
        \begin{tikzpicture}[scale=2]
            \clip(-0.5, -0.25) rectangle (1.25, 1.30);
            
            \drawContRelax{1.0}             
            \drawSplitHull{0.0}{1.0}{1.0}   
            \drawXbar{0.707}{0.707}         
            
            \draw [line width=1pt, color=red] (-0.50, 1.50) -- (1.50, -0.50);
            
        \end{tikzpicture}
        }
        \hfill
        \subfloat[]{
        \label{fig:ex:polar:infeasible}
        \begin{tikzpicture}[scale=2]
            \clip(-0.5, -0.25) rectangle (1.25, 1.30);
            
            \drawContRelax{0.9}             
            \fill[orange!50!white] (90:0.9) arc (90:270:0.9);  
            
            \drawXbar{0.6364}{0.6364}
            \draw [line width=1pt, color=red] (0.0, -0.50) -- (0.0, 1.20);
            
            
            
        \end{tikzpicture}
        }
        
        \caption{Split cuts (in red) obtained with the polar normalization.}
        \label{fig:ex:polar}
    \end{figure}
    
    \begin{table}
        \centering
        \caption{CGCP statistics for Example \ref{ex:normalization} and polar normalization}
        \label{tab:CGCP:polar}
        \begin{tabular}{ccrrrrrrrr}
        \toprule
            & Iter & $\| \alpha \|$ & $\| u_{1}\|$ & $\| u_{2}\|$ & $\|\lambda_{1}\|$ & $\|\lambda_{2}\|$ & $\| v_{1}\|$ & $\| v_{2}\|$ \\
        \midrule
            \ref{ex:enum:well-posed}  & $ 8$ &     1.2 & $     0.0 $ & $     1.1 $ &     1.1 &     2.7 &     0.5 &     1.2\\
            \ref{ex:enum:ill-posed}  & $66^{*}$ &     1.2 & $     0.0 $ & $  9912.6 $ &     1.0 & 14019.5 &     0.7 &  9912.6\\
            \ref{ex:enum:infeasible}   & $ 8$ &     1.6 & $     0.0 $ & $    20.5 $ &     0.0 &    28.9 &     1.6 &    18.8\\
        \bottomrule
        \end{tabular}\\
        $^{*}$: slow progress
    \end{table}
    
    Similar to the $\alpha$-normalization, if all disjunctive terms are non-empty, then cuts obtained with the polar normalization are also supporting hyperplanes of the disjunctive hull, as expressed by Theorem \ref{thm:polar:supportingHyperplane}.
    
    \begin{theorem}
    \label{thm:polar:supportingHyperplane}
        Assume that the CGCP \eqref{eq:CGCP:polar} and the MCP \eqref{eq:MCP:polar} are solvable, and let $(\alpha, \beta, u, \lambda, v)$ and $(\eta, y, z)$ be corresponding optimal solutions.
        Then, 
        \begin{align*}
            \alpha^{T}(\bar{x} + \eta \gamma) = \beta.
        \end{align*}
    \end{theorem}
    \begin{proof}
        By conic strong duality, we have $\alpha^{T}\bar{x} - \beta = -\eta$.
        If the optimal value of the CGCP is $0$, then the result is trivial.
        Similarly, if $\alpha^{T}\gamma \leq 0$, then the optimal value of the CGCP must be $0$ and the result is trivial.
        
        We now assume that $\alpha^{T}\gamma > 0$ and $\opt(CGCP) < 0$.
        Since $(\alpha, \beta, u, \lambda, v)$ is optimal for the CGCP, we must have $\alpha^{T}\gamma = 1$.
        Then,
        \begin{align*}
            \alpha^{T}\bar{x} - \beta 
            &= -\eta\\
            &= -(\alpha^{T}\gamma) \eta,
        \end{align*}
        i.e., $\alpha^{T}(\bar{x} + \eta \gamma) = \beta$.
        \qed
    \end{proof}

\subsection{Standard normalization}
\label{sec:normalization:standard}

    One of the most common normalizations in the MILP setting is the so-called \emph{standard normalization} \cite{Balas1979,Balas2003,Fischetti2011}.
    Here, we introduce its conic generalization 
    \begin{align}
        \label{eq:norm:standard}
        \sum_{h} \knorm{\rho}{\lambda_{h}} + \knorm{\sigma_{h}}{v_{h}} \leq 1,
    \end{align}
    where $\rho \in \intr \K$ and $\sigma_{h} \in \intr \Q_{h}$.
    When all cones $\K, \Q_{1}, ..., \Q_{H}$ are non-negative orthants, \eqref{eq:norm:standard} indeed generalizes both the standard normalization for MILP and its \emph{Euclidean normalization} variant \cite{Fischetti2011}, by setting $\rho$ and $\sigma$ appropriately.
    
    The separation problem then writes
    \begin{subequations}
    \label{eq:CGCP:standard}
    \begin{align}
        \label{eq:CGCP:standard:obj}
        \min_{\alpha, \beta, u, v, \lambda} \ \ \ & \alpha^{T} \bar{x} - \beta\\
        s.t. \ \ \ 
        & \label{eq:CGCP:standard:f1} \alpha = A^{T}u_{h} + \lambda_{h} + D_{h}^{T} v_{h}, && \forall h,\\
        & \label{eq:CGCP:standard:f2} \beta  \leq b^{T}u_{h} + d_{h}^{T}v_{h}, && \forall h,\\
        & \label{eq:CGCP:standard:domain} (u_{h}, \lambda_{h}, v_{h}) \in \R^{m} \times \K^{*} \times \Q_{h}^{*}, && \forall h,\\
        & \label{eq:CGCP:standard:normalization} \sum_{h} \knorm{\rho}{\lambda_{h}} + \knorm{\sigma_{h}}{v_{h}}  \leq 1.
    \end{align}
    \end{subequations}
    Note that \eqref{eq:CGCP:standard:normalization} consists of a single linear inequality in the $\lambda, v$ space, and that it explicitly bounds their magnitude.
    
    Up to a change of sign in the objective value, the MCP is
    \begin{subequations}
    \label{eq:MCP:standard}
    \begin{align}
        \label{eq:MCP:standard:obj}
        \min_{y, z, \eta} \ \ \ & \eta\\
        s.t. \ \ \
        \label{eq:MCP:standard:friends}  & \sum_{h} y_{h} = \bar{x},\\
        \label{eq:MCP:standard:convex}   & \sum_{h} z_{h} = 1,\\
        \label{eq:MCP:standard:constr1}  & A y_{h} = z_{h} b, & \forall h,\\
        \label{eq:MCP:standard:disj}     & D_{h} y_{h} + \eta \sigma_{h} \Kgeq{\Q_{h}} z_{h} d_{h}, & \forall h,\\
        \label{eq:MCP:standard:domain}    & (y_{h} + \eta \rho, z_{h}) \in \K \times \R_{+}, & \forall h,\\
        \label{eq:MCP:standard:nonneg}   & \eta \geq 0.
    \end{align}
    \end{subequations}
    The dual counterpart of the standard normalization corresponds to penalizing the violation of the original conic constraints  \eqref{eq:MCP:domain} and \eqref{eq:MCP:disj}, through the artificial slack variable $\eta$ in \eqref{eq:MCP:standard:domain} and \eqref{eq:MCP:standard:disj}.
    As shown in Theorem \ref{thm:standard:StrongDuality}, this ensures that the MCP is strongly feasible.
    
    \begin{theorem}
        \label{thm:standard:StrongDuality}
        The CGCP \eqref{eq:CGCP:standard} and the MCP \eqref{eq:MCP:standard}  are strongly feasible.
    \end{theorem}
    \begin{proof}
        It was shown in Section \ref{sec:separation:CGCP} that the un-normalized CGCP \eqref{eq:CGCP} admits a strongly feasible solution.
        Scaling this point appropriately yields a strongly feasible solution for the CGCP \eqref{eq:CGCP:standard}.
        
        We now show that the MCP \eqref{eq:MCP:trivial:constr1} is strongly feasible.
        Let $y_{h} = \frac{1}{H} \bar{x}$ and $z_{h} = \frac{1}{H}$, $\forall h$.
        Then, since $\K, \Q_{1}, ..., \Q_{H}$ are proper cones, there exists $\eta_{h} \geq 0$ and $\tau_{h} \geq 0$ such that
        \begin{align*}
            \forall h, \ &y_{h} + \eta_{h} \rho \Kgt{\K} 0,\\
            \forall h, \ &D_{h} y_{h} + \tau_{h} \sigma_{h} \Kgt{\Q_{h}} d_{h}.
        \end{align*}
        Letting $\eta = \max (\eta_{1}, ..., \eta_{H}, \tau_{1}, ..., \tau_{H})$ then yields a strongly feasible point $(y, z, \eta)$, which concludes the proof.
        \qed
    \end{proof}
    
    A direct consequence of Theorem \ref{thm:standard:StrongDuality} is that both the CGCP \eqref{eq:CGCP:standard} and the MCP \eqref{eq:MCP:standard} are solvable, and that conic strong duality holds.
    Importantly, aside from the assumption $A \bar{x} = b$, this result does not depend on $\bar{x}$, nor on the well-posedness of individual disjunctions.
    
    Split cuts obtained with the standard normalization are illustrated in Figure \ref{fig:ex:standard}, and the corresponding CGCP statistics are reported in Table \ref{tab:CGCP:standard}.
    On the one hand, Table \ref{tab:CGCP:standard} illustrate the good numerical behavior of the CGCP \eqref{eq:CGCP:standard}, a direct consequence of having enforced strong feasibility of the MCP.
    On the other hand, the obtained cuts are not as strong as the ones obtained with the $\alpha$ or polar normalizations.
    Indeed, in general, the cut obtained with the standard normalization may not be a supported hyperplane of the disjunctive hull.

    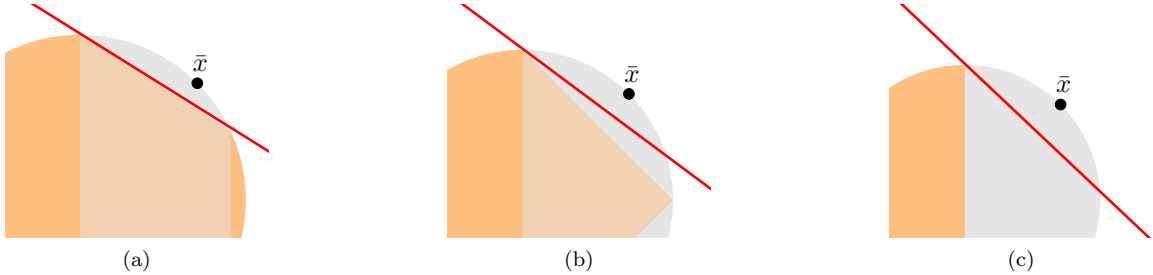
\begin{figure}
        \centering
        \subfloat[]{
        \label{fig:ex:standard:well-posed}
        \begin{tikzpicture}[scale=2]
            \clip(-0.5, -0.25) rectangle (1.25, 1.30);
            
            \drawContRelax{1.1}             
            \drawSplitHull{0.0}{1.0}{1.1}   
            \drawXbar{0.778}{0.778}         
            
            \draw [line width=1pt, color=red] (-0.50, 1.42) -- (1.50, 0.17);
            
        \end{tikzpicture}
        }
        \hfill
        \subfloat[]{
        \label{fig:ex:standard:ill-posed}
        \begin{tikzpicture}[scale=2]
            \clip(-0.5, -0.25) rectangle (1.25, 1.30);
            
            \drawContRelax{1.0}             
            \drawSplitHull{0.0}{1.0}{1.0}   
            \drawXbar{0.707}{0.707}         
            
            \draw [line width=1pt, color=red] (-0.50, 1.38) -- (1.50, -0.11);
            
        \end{tikzpicture}
        }
        \hfill
        \subfloat[]{
        \label{fig:ex:standard:infeasible}
        \begin{tikzpicture}[scale=2]
            \clip(-0.5, -0.25) rectangle (1.25, 1.30);
            
            \drawContRelax{0.9}             
            \fill[orange!50!white] (90:0.9) arc (90:270:0.9);  
            
            \drawXbar{0.6364}{0.6364}  
            \draw [line width=1pt, color=red] (-0.50, 1.38) -- (1.50, -0.51);
            
        \end{tikzpicture}
        }
        
        \caption{%
            Split cuts (in red) obtained with the standard normalization.
        }
        \label{fig:ex:standard}
    \end{figure}
    
    \begin{table}
        \centering
        \caption{CGCP statistics for Example \ref{ex:normalization} and standard normalization}
        \label{tab:CGCP:standard}
        \begin{tabular}{ccrrrrrrrr}
        \toprule
            & Iter & $\| \alpha \|$ & $\| u_{1}\|$ & $\| u_{2}\|$ & $\|\lambda_{1}\|$ & $\|\lambda_{2}\|$ & $\| v_{1}\|$ & $\| v_{2}\|$ \\
        \midrule
            \ref{ex:enum:well-posed}  & $ 8$ &     0.4 & $     0.0 $ & $     0.2 $ &     0.3 &     0.6 &     0.1 &     0.2\\
            \ref{ex:enum:ill-posed}  & $ 8$ &     0.3 & $     0.0 $ & $     0.2 $ &     0.3 &     0.6 &     0.1 &     0.2\\
            \ref{ex:enum:infeasible}  & $ 7$ &     0.3 & $     0.0 $ & $     0.3 $ &     0.3 &     0.6 &     0.1 &     0.2\\
        \bottomrule
        \end{tabular}
    \end{table}

\subsection{Trivial normalization}
\label{sec:nomralization:trivial}

    The trivial normalization is obtained by setting $\rho = 0$ in the standard normalization.
    The separation problem then writes
    \begin{subequations}
    \label{eq:CGCP:trivial}
    \begin{align}
        \label{eq:CGCP:trivial:obj}
        \min_{\alpha, \beta, u, v, \lambda} \ \ \ & \alpha^{T} \bar{x} - \beta\\
        s.t. \ \ \ 
        & \label{eq:CGCP:trivial:f1} \alpha = A^{T}u_{h} + \lambda_{h} + D_{h}^{T} v_{h}, && \forall h,\\
        & \label{eq:CGCP:trivial:f2} \beta  \leq b^{T}u_{h} + d_{h}^{T}v_{h}, && \forall h,\\
        & \label{eq:CGCP:trivial:domain} (u_{h}, \lambda_{h}, v_{h}) \in \R^{m} \times \K^{*} \times \Q_{h}^{*}, && \forall h,\\
        & \label{eq:CGCP:trivial:normalization} \sum_{h} \knorm{\sigma_{h}}{v_{h}}  \leq 1.
    \end{align}
    \end{subequations}
    Note that, for a split disjunction $(-\pi^{T}x \geq -\pi_{0}) \vee (\pi^{T}x \geq \pi_{0} +1)$ and $\sigma_{h} = 1$, \eqref{eq:CGCP:trivial:normalization} reduces to
    \begin{align}
        \label{eq:trivial:MILP}
        v_{1} + v_{2} \leq 1,
    \end{align}
    which is the so-called trivial normalization \cite{Fischetti2011} in the MILP setting.
    In particular, Gomory Mixed-Integer cuts correspond to optimal solutions of the CGLP with trivial normalization.
    
    Up to a change of sign in the objective value, the MCP writes
    \begin{subequations}
    \label{eq:MCP:trivial}
    \begin{align}
        \label{eq:MCP:trivial:obj}
        \min_{y, z, \eta} \ \ \ & \eta\\
        s.t. \ \ \
        \label{eq:MCP:trivial:friends}  & \sum_{h} y_{h} = \bar{x},\\
        \label{eq:MCP:trivial:convex}   & \sum_{h} z_{h} = 1,\\
        \label{eq:MCP:trivial:constr1}  & A y_{h} = z_{h} b, & \forall h,\\
        \label{eq:MCP:trivial:domain}    & (y_{h}, z_{h}) \in \K \times \R_{+}, & \forall h,\\
        \label{eq:MCP:trivial:disj}     & D_{h} y_{h} + \eta \sigma_{h} \Kgeq{\Q_{h}} z_{h} d_{h}, & \forall h,\\
        \label{eq:MCP:trivial:nonneg}   & \eta \geq 0.
    \end{align}
    \end{subequations}
    
    \begin{theorem}
    \label{thm:MCP:trivial:feasibility}
        The MCP \eqref{eq:MCP:trivial} is
        \begin{enumerate}
            \item strongly infeasible if and only if $\bar{x}$ is infeasible for $(CP)$;
            \item strongly feasible if and only if $\bar{x}$ is strongly feasible for $(CP)$;
            \item weakly feasible if and only if $\bar{x}$ is weakly feasible for $(CP)$.
        \end{enumerate}
    \end{theorem}
    \begin{proof}
        Since $\bar{x}$ is either infeasible, strongly feasible or weakly feasible for $(CP)$, and that these are mutually exclusive alternatives, it suffices to prove that the above conditions are sufficient;  necessity follows immediately by contraposition.
        For simplicity, we assume that $\K$ is an irreducible, non-polyhedral cone.
        The general case is treated similarly.\\
        
        \textbf{1.} Assume $\bar{x} \notin \C$, i.e., $\bar{x} \notin \K$ since we assumed that $A\bar{x} = b$.
        Consequently, there exists $s \in \K^{*}$ such that $s^{T}\bar{x} < 0$ and, since $\K$ is a proper cone, we can assume without loss of generality that $s \in \intr \K^{*}$.
        Thus, setting
        \begin{align*}
            (\alpha, \beta) &= (s, 0)\\
            (u_{h}, \lambda_{h}, v_{h}) &= (0, s, 0), &\forall h,
        \end{align*}
        yields a strongly feasible solution of the CGCP \eqref{eq:CGCP:trivial}, which is also an unbounded ray.
        Therefore, the MCP is strongly infeasible.
        
        \textbf{2, 3.}
        We first show that, if $\bar{x}$ is feasible for $(CP)$, then the MCP is feasible.
        Let
        \begin{align*}
            \tilde{y}_{h} &= H^{-1} \bar{x}, & \forall h,\\
            \tilde{z}_{h} &= H^{-1},         & \forall h.
        \end{align*}
        It is then immediate that $A \tilde{y}_{h} = \tilde{z}_{h} b$, $\tilde{y}_{h} \in \K$, and $\sum_{h} \tilde{z}_{h} = 1$.
        Then, since $\sigma_{h} \in \intr \Q_{h}$, there exists $\eta_{h} \geq 0$ such that \begin{align*}
            D_{h} y_{h} + \eta_{h} \sigma_{h} \Kgt{\Q_{h}} z_{h} d_{h}.
        \end{align*}
        Letting $\tilde{\eta} = \max_{h} \{ \eta_{h}\}$, it follows that $(\tilde{y}, \tilde{z}, \tilde{\eta})$ is feasible for the MCP.\\
        
        
        If $\bar{x} \in \intr \K$, then we also have $\tilde{y}_{h} \in \intr \K$, and thus $(\tilde{y}, \tilde{z}, \tilde{\eta})$ is strongly feasible for the MCP, which proves \textbf{2.}
        Reciprocally, let $(y, z, \eta)$ be a strongly feasible solution of MCP.
        In particular, we have $y_{h} \in \intr \K$.
        Then,
        \begin{align*}
            \bar{x} = \sum_{h} y_{h} \in \intr \K,
        \end{align*}
        i.e., $\bar{x}$ is strongly feasible for $(CP)$, thereby proving \textbf{3.} by contraposition.
        \qed
    \end{proof}

    Theorem \ref{thm:MCP:trivial:feasibility} motivates the following remarks.
    First, case \textbf{1.} typically arises in the context of outer-approximation algorithms, wherein fractional points generally violate non-linear conic constraints.
    Then, the CGCP \eqref{eq:CGCP:trivial} is unbounded, and the normalization \eqref{eq:CGCP:trivial:normalization} imposes $\knorm{\sigma_{h}}{v_{h}} = 0, \forall h$ in any unbounded ray, i.e., $v_{h} = 0$ and the obtained cut is always a trivial inequality.
    Thus, the trivial normalization is not suited for use within outer approximation-based algorithms.
    
    Second, conic-infeasible points are not encountered in non-linear branch-and-bound algorithms.
    However, solving $(CP)$ yields a fractional point $\bar{x} \in \partial \C$ that is weakly feasible, unless all non-linear conic constraints are inactive at the optimum.
    Although the current fractional point $\bar{x}$ may become strongly feasible after several rounds of cuts, or deeper in the branch-and-bound tree, our experience is that case \textbf{2.} rarely occurs in practice.
    
    Third, case \textbf{3.} corresponds to the setting of Example \ref{ex:normalization}.
    Split cuts obtained with the trivial normalization are displayed in Figure \ref{fig:ex:trivial}, and the corresponding CGCP statistics are reported in Table \ref{tab:CGCP:trivial}.
    Remarkably, all three cuts appear to be $\K^{*}$ cuts, while numerical issues, namely, slow convergence, are systematically encountered.
    Here, the MCP \eqref{eq:MCP:trivial} is weakly feasible and the CGCP \eqref{eq:CGCP:trivial} is bounded but not solvable.
    Thus, there exists a diverging sequence of close-to-optimal solution, thereby causing slow convergence.
    In addition, since \eqref{eq:CGCP:trivial:normalization} bounds the value of $v_{1}, v_{2}$, the iterates become equivalent to a $\K^{*}$ cut as the magnitude of $\lambda$ increases, which explains the cuts obtained in Figure \ref{fig:ex:trivial}.
    
    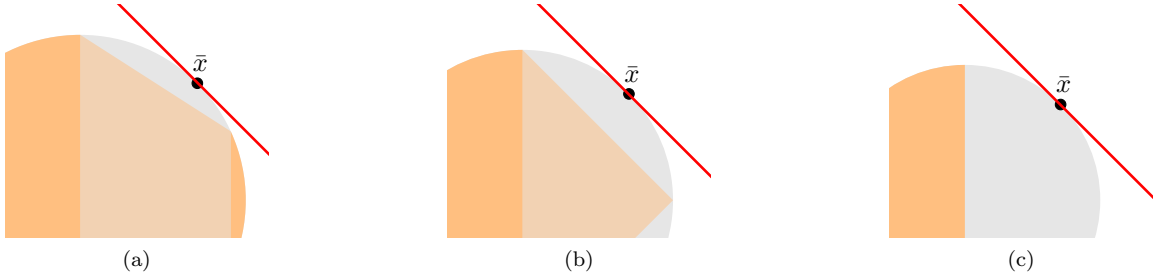
\begin{figure}
        \centering
        \subfloat[]{
        \label{fig:ex:trivial:well-posed}
        \begin{tikzpicture}[scale=2]
            \clip(-0.5, -0.25) rectangle (1.25, 1.30);
            
            \drawContRelax{1.1}             
            \drawSplitHull{0.0}{1.0}{1.1}   
            \drawXbar{0.778}{0.778}         
            
            \draw [line width=1pt, color=red] (-0.50, 2.06) -- (1.50, 0.06);
            
        \end{tikzpicture}
        }
        \hfill
        \subfloat[]{
        \label{fig:ex:trivial:ill-posed}
        \begin{tikzpicture}[scale=2]
            \clip(-0.5, -0.25) rectangle (1.25, 1.30);
            
            \drawContRelax{1.0}             
            \drawSplitHull{0.0}{1.0}{1.0}   
            \drawXbar{0.707}{0.707}         
            
            \draw [line width=1pt, color=red] (-0.50, 1.91) -- (1.50, -0.09);
            
        \end{tikzpicture}
        }
        \hfill
        \subfloat[]{
        \label{fig:ex:trivial:infeasible}
        \begin{tikzpicture}[scale=2]
            \clip(-0.5, -0.25) rectangle (1.25, 1.30);
            
            \drawContRelax{0.9}             
            \fill[orange!50!white] (90:0.9) arc (90:270:0.9);  
            
            \drawXbar{0.6364}{0.6364}  
            \draw [line width=1pt, color=red] (-0.50, 1.77) -- (1.50, -0.23);
            
        \end{tikzpicture}
        }
        
        \caption{%
            Split cuts (in red) obtained with the trivial normalization.
        }
        \label{fig:ex:trivial}
    \end{figure}
    
    \begin{table}
        \centering
        \caption{CGCP statistics for Example \ref{ex:normalization} and trivial normalization}
        \label{tab:CGCP:trivial}
        \begin{tabular}{ccrrrrrrrr}
        \toprule
            & Iter & $\| \alpha \|$ & $\| u_{1}\|$ & $\| u_{2}\|$ & $\|\lambda_{1}\|$ & $\|\lambda_{2}\|$ & $\| v_{1}\|$ & $\| v_{2}\|$ \\
        \midrule
            \ref{ex:enum:well-posed}  & $37^{*}$ &  6189.7 & $     0.0 $ & $     0.7 $ &  6189.6 &  6190.6 &     0.2 &     0.8\\
            \ref{ex:enum:ill-posed}  & $49^{*}$ &  6814.2 & $     0.0 $ & $     0.7 $ &  6814.0 &  6815.0 &     0.3 &     0.7\\
            \ref{ex:enum:infeasible}  & $72^{*}$ &  6364.4 & $     0.0 $ & $     0.7 $ &  6364.2 &  6365.2 &     0.4 &     0.6\\
        \bottomrule
        \end{tabular}\\
        $^{*}$: slow progress
    \end{table}
    
    Note that the CGCP \eqref{eq:CGCP:trivial} may be solvable even though $\bar{x}$ is weakly feasible.
    However, our experience suggests that this rarely happens, and that most cases are similar to Example \ref{ex:normalization}, leading to numerical issues and weak cuts.

\subsection{Uniform normalization}
\label{sec:normalization:uniform}

    The uniform normalization is obtained by setting $\sigma_{h} = 0$ in the standard normalization.
    
    The CGCP then writes
    \begin{subequations}
    \label{eq:CGCP:uniform}
    \begin{align}
        \label{eq:CGCP:uniform:obj}
        \min_{\alpha, \beta, u, v, \lambda} \ \ \ & \alpha^{T} \bar{x} - \beta\\
        s.t. \ \ \ 
        & \label{eq:CGCP:uniform:f1} \alpha = A^{T}u_{h} + \lambda_{h} + D_{h}^{T} v_{h}, && \forall h,\\
        & \label{eq:CGCP:uniform:f2} \beta  \leq b^{T}u_{h} + d_{h}^{T}v_{h}, && \forall h,\\
        & \label{eq:CGCP:uniform:domain} (u_{h}, \lambda_{h}, v_{h}) \in \R^{m} \times \K^{*} \times \Q_{h}^{*}, && \forall h,\\
        & \label{eq:CGCP:uniform:normalization} \sum_{h} \knorm{\rho}{\lambda_{h}}  \leq 1,
    \end{align}
    \end{subequations}
    and, up to a change of sign in the objective, the MCP is
    \begin{subequations}
    \label{eq:MCP:uniform}
    \begin{align}
        \label{eq:MCP:uniform:obj}
        \min_{y, z, \eta} \ \ \ & \eta\\
        s.t. \ \ \
        \label{eq:MCP:uniform:friends}  & \sum_{h} y_{h} = \bar{x},\\
        \label{eq:MCP:uniform:convex}   & \sum_{h} z_{h} = 1,\\
        \label{eq:MCP:uniform:equality}  & A y_{h} = z_{h} b, & \forall h,\\
        \label{eq:MCP:uniform:domain}    & (y_{h} + \eta \rho, z_{h}) \in \K \times \R_{+}, & \forall h,\\
        \label{eq:MCP:uniform:disj}     & D_{h} y_{h} \Kgeq{\Q_{h}} z_{h} d_{h}, & \forall h,\\
        \label{eq:MCP:uniform:nonneg}   & \eta \geq 0.
    \end{align}
    \end{subequations}

    As illustrated by Example \ref{ex:MCP:uniform}, in general, the MCP \eqref{eq:MCP:uniform} may not be feasible.
    
    \begin{example}
    \label{ex:MCP:uniform}
        Let
        \begin{align*}
            \C = \left\{ (x_{1}, x_{2}) \in \R_{+}^{2} \ | \ x_{1} + x_{2} = 1 \right\},
        \end{align*}
        and consider the disjunction
        \begin{align*}
            \left\{
                x_{1} \geq 0, -x_{1} \geq 0
            \right\}
            \vee
            \left\{
                x_{1} \geq 1, -x_{1} \geq -1
            \right\}.
        \end{align*}
        Constraints \eqref{eq:MCP:uniform:equality} and \eqref{eq:MCP:uniform:disj} first yield
        \begin{align*}
            y_{1} = \begin{pmatrix} 0\\ z_{1} \end{pmatrix}, \ \ \
            y_{2} &= \begin{pmatrix} z_{2}\\ 1 - z_{2} \end{pmatrix},
        \end{align*}
        which, combined with \eqref{eq:MCP:uniform:convex} and \eqref{eq:MCP:uniform:friends},
        yields $\bar{x} = (1-z_{1}, z_{1})$ for $0 \leq z_{1} \leq 1$.
        Therefore, if $\bar{x} = (-1, 2)$, then the MCP \eqref{eq:MCP:uniform} is infeasible.
    \end{example}
    
    Nevertheless, the following results demonstrate that, for certain classes of disjunctions, namely split disjunctions, strong feasibility in the MCP is guaranteed.
    
    \begin{lemma}
    \label{lem:MCP:uniform:feas}
        If
        \begin{align*}
            \bar{x} \in \conv \left(
                \bigcup_{h} \left\{ x \ \middle| \ Ax = b, D_{h} x \Kgeq{\Q_{h}} d_{h} \right\}
            \right),
        \end{align*}
        then the MCP \eqref{eq:MCP:uniform}  is feasible.
    \end{lemma}
    \begin{proof}
        Immediate from Theorem \ref{thm:ConvexHull}.
        \qed
    \end{proof}
    
    \begin{lemma}
    \label{lem:MCP:uniform:polyhedral}
        Assume that $\forall h, \ \Q_{h}$ is polyhedral.
        Then, the MCP \eqref{eq:MCP:uniform} is strongly feasible if and only if it is feasible.
    \end{lemma}
    \begin{proof}
        Strong feasibility implies feasibility.
        Reciprocally, if $(x, y, z, \eta)$ is feasible for the MCP \eqref{eq:MCP:uniform}, then $(x, y, z, \eta + \epsilon)$ is strongly feasible for any $\epsilon > 0$.
        \qed
    \end{proof}
    
    \begin{theorem}
    \label{thm:MCP:uniform:split}
        If $\X \neq \emptyset$ and $\D$ is a split disjunction, i.e.,
        \begin{align*}
            \D =
            \left\{
                x
                \ \middle| \ 
                \begin{array}{l}
                    A x = b, x \in \K\\
                    \pi^{T}x \leq \pi_{0}
                \end{array}
            \right\}
            \cup
            \left\{
                x
                \ \middle| \ 
                \begin{array}{l}
                    A x = b, x \in \K\\
                    \pi^{T}x \geq \pi_{0} + 1
                \end{array}
            \right\}
            ,
        \end{align*}
        then the MCP \eqref{eq:MCP:uniform} is strongly feasible.
    \end{theorem}
    \begin{proof}
        Let
        \begin{align*}
            \mathcal{S} = \conv \left(
            \left\{ x \ \middle| \ Ax = b, \pi^{T}x \leq \pi_{0} \right\}
            \cup \left\{ x \ \middle| \ Ax = b, \pi^{T}x \geq \pi_{0} + 1 \right\}
            \right).
        \end{align*}
        
        First, assume there exists $\xi \in \R^{n}$ such that $A \xi = 0$ and $\pi^{T}\xi \neq 0$; without loss of generality, we can assume that $\pi^{T}\xi = 1$.
        Then, let $t \geq 0$ such that
        \begin{align*}
            \pi^{T} (\bar{x} - t \xi) &\leq \pi_{0},\\
            \pi^{T} (\bar{x} + t \xi) &\geq \pi_{0} + 1.
        \end{align*}
        In particular, $A (\bar{x} \pm t \xi) = b$ and $\bar{x} = \frac{1}{2}(\bar{x} + t \xi) + \frac{1}{2}(\bar{x} - t \xi)$.
        Thus,  $\bar{x} \in \mathcal{S}$, and it follows from Lemma \ref{lem:MCP:uniform:feas} that the MCP \eqref{eq:MCP:uniform} is feasible.
        
        Now assume that $\forall \xi \in \ker(A), \pi^{T}\xi = 0$.
        On the one hand, if $\pi_{0} < \pi^{T}\bar{x} < \pi_{0} + 1$, then $\mathcal{S} = \emptyset$ and thus $\X$ is empty, which would contradict the $\X \neq \emptyset$ assumption.
        Thus, either $\pi^{T}\bar{x} \leq \pi_{0}$ or $\pi^{T}\bar{x} \geq \pi_{0} +1$, i.e., $\bar{x} \in \mathcal{S}$ and the MCP is feasible by Lemma \ref{lem:MCP:uniform:feas}.
        Note that this latter case would never occur in practice, since one would always consider a split such that $\pi_{0} < \pi^{T}\bar{x} < \pi_{0} + 1$.
        
        The result then follows from Lemma \ref{lem:MCP:uniform:polyhedral}.
        \qed
    \end{proof}
    
    Similar to the standard normalization, a consequence of Theorem \ref{thm:MCP:uniform:split} is that, when considering split disjunctions, the CGCP \eqref{eq:CGCP:uniform} and the MCP \eqref{eq:MCP:uniform} are solvable with identical objective value.
    This is confirmed by the results of Table \ref{tab:CGCP:polar}, which reports statistics for the CGCP in Example \ref{ex:normalization}: no numerical issue is encountered.
    The corresponding split cuts are displayed in Figure \ref{fig:ex:uniform}, and are similar to the cuts obtained with the standard normalization.
    Likewise, cuts obtained with the uniform normalization are not, in general, supporting hyperplanes of the disjunctive hull.
    
    \begin{figure}
        \centering
        \subfloat[]{
        \label{fig:ex:uniform:well-posed}
        \begin{tikzpicture}[scale=2]
            \clip(-0.5, -0.25) rectangle (1.25, 1.30);
            
            \drawContRelax{1.1}             
            \drawSplitHull{0.0}{1.0}{1.1}   
            \drawXbar{0.778}{0.778}         
            
            \draw [line width=1pt, color=red] (-0.50, 1.41) -- (1.50, 0.16);
            
        \end{tikzpicture}
        }
        \hfill
        \subfloat[]{
        \label{fig:ex:uniform:ill-posed}
        \begin{tikzpicture}[scale=2]
            \clip(-0.5, -0.25) rectangle (1.25, 1.30);
            
            \drawContRelax{1.0}             
            \drawSplitHull{0.0}{1.0}{1.0}   
            \drawXbar{0.707}{0.707}         
            
            \draw [line width=1pt, color=red] (-0.50, 1.39) -- (1.50, -0.16);
            
        \end{tikzpicture}
        }
        \hfill
        \subfloat[]{
        \label{fig:ex:uniform:infeasible}
        \begin{tikzpicture}[scale=2]
            \clip(-0.5, -0.25) rectangle (1.25, 1.30);
            
            \drawContRelax{0.9}             
            \fill[orange!50!white] (90:0.9) arc (90:270:0.9);  
            
            \drawXbar{0.6364}{0.6364}  
            \draw [line width=1pt, color=red] (-0.50, 1.43) -- (1.50, -0.68);
            
        \end{tikzpicture}
        }
        
        \caption{%
            Split cuts obtained with the uniform normalization.
            The continuous relaxation is in gray, the split hull in orange, and the obtained cut is in red.
        }
        \label{fig:ex:uniform}
    \end{figure}
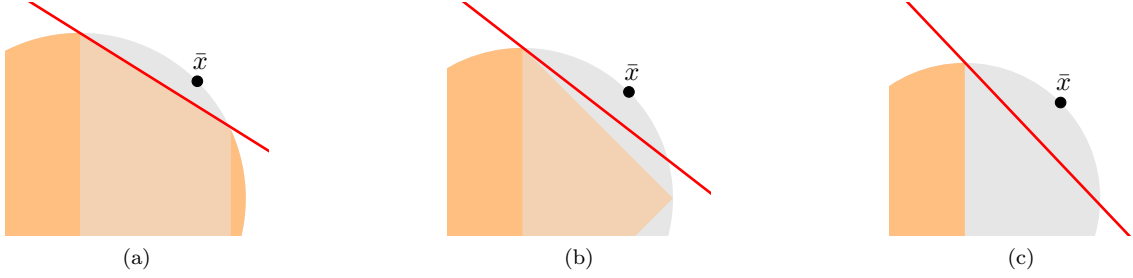
    
    \begin{table}
        \centering
        \caption{CGCP statistics for Example \ref{ex:normalization} and uniform normalization}
        \label{tab:CGCP:uniform}
        \begin{tabular}{ccrrrrrrrr}
        \toprule
            & Iter & $\| \alpha \|$ & $\| u_{1}\|$ & $\| u_{2}\|$ & $\|\lambda_{1}\|$ & $\|\lambda_{2}\|$ & $\| v_{1}\|$ & $\| v_{2}\|$ \\
        \midrule
            \ref{ex:enum:well-posed}  & $ 8$ &     0.5 & $     0.0 $ & $     0.3 $ &     0.5 &     0.9 &     0.2 &     0.3\\
            \ref{ex:enum:ill-posed}  & $ 8$ &     0.5 & $     0.0 $ & $     0.4 $ &     0.4 &     1.0 &     0.2 &     0.4\\
            \ref{ex:enum:infeasible}  & $ 8$ &     0.4 & $     0.0 $ & $     0.5 $ &     0.3 &     1.1 &     0.3 &     0.5\\
        \bottomrule
        \end{tabular}
    \end{table}

\section{Separating conic-infeasible points}
\label{sec:infeas}

When $(MICP)$ \eqref{eq:MICP} is solved by outer-approximation, the fractional point $\bar{x}$ may not satisfy all conic constraints.
In preliminary experiments, wherein lift-and-project cuts were separated by rounds in a callback, a large proportion --often higher than $90 \%$-- of the cuts yielded by the CGCP turned out to be $\K^{*}$ cuts, which is obviously detrimental to performance.
To the best of our knowledge, despite the popularity and performance of outer-approximation algorithms, this behavior has not been studied in the literature.
Therefore, in this section, we will assume that $A\bar{x} = b$, but $\bar{x} \notin \K$.

Let $(\alpha, \beta, u, \lambda, v)$ be a solution of the CGCP, and assume that $v_{h} = 0$ for some $h \in \{1, ..., H\}$.
Thus, we have
\begin{align}
    \alpha & = A^{T}u_{h} + \lambda_{h},\\
    \beta  & \leq b^{T}u_{h},
\end{align}
and $\alpha^{T}x \geq \beta$ is a trivial inequality.
In addition, as noted in Section \ref{sec:separation:CGCP}, the inequality $\lambda_{h}^{T}x \geq 0$ has the same violation, and cuts off the same portion of the continuous relaxation as $\alpha^{T}x \geq \beta$.
This observation, which does not depend on the normalization condition, allows the \emph{a posteriori} detection of $\K^{*}$ cuts, by checking the value of the $v$ multipliers.

Once a $\K^{*}$ cut is identified, it can be disaggregated.
Assume that $\K = \K_{1} \times ... \times \K_{N}$; correspondingly, for $\lambda \in \K^{*}$, we write $\lambda = (\lambda_{1}, ..., \lambda_{N})$ where each $\lambda_{i} \in \K^{*}_{i}$.
Then, the $\K^{*}$ cut $\lambda^{T}x \geq 0$ is disaggregated as
\begin{align}
    \lambda_{i}^{T}x_{i} \geq 0, \ \ \ i = 1, ..., N.
\end{align}
This yields more numerous, but sparser, $\K^{*}$ cuts, and results in tighter polyhedral approximations which, in turn, improves the performance of outer-approximation algorithms \cite{coey2018outer}.

We now derive sufficient conditions that provide an \emph{a priori} indication that a $\K^{*}$ cut will be generated.
Unless stated otherwise, we only consider the CGCP with standard normalization.
Define
\begin{align}
    \label{eq:infeasible:etabar}
    \bar{\eta} &= \min_{\eta \geq 0} \left\{ \eta \ \middle| \ \bar{x} + \eta \rho \in \K \right\},\\
    \label{eq:infeasible:taubar}
    \bar{\tau}_{h} &= \min_{\tau \geq 0} \left\{ \tau \ \middle| \ D_{h} \bar{x} + \tau \sigma_{h} \Kgeq{\Q_{h}} d_{h} \right\}, & \forall h,
\end{align}
and let $\xi = \bar{x} + \bar{\eta} \rho $.

\begin{lemma}
    Assume there exists an optimal solution of the CGCP for which $v_{h} = 0, \forall h$.
    Then, the optimal value of the CGCP is $-H^{-1} \bar{\eta}$, and there exist a CGCP-optimal solution of the form $(\lambda_{0}, 0, 0, \lambda_{0}, 0)$,
    with $\lambda_{0}^{T}\xi = 0$.
\end{lemma}
\begin{proof}
    Let $(\alpha, \beta, u, \lambda, 0)$ be such an optimal solution, and denote by $\delta$ its objective value.
    In particular, we have
    \begin{align*}
        \alpha &= A^{T} u_{h} + \lambda_{h}, && \forall h,\\
        \beta &\leq b^{T}u_{h}, && \forall h.
    \end{align*}
    Let $\bar{u} = \frac{1}{H} \sum_{h} u_{h}$ and $\bar{\lambda} = \frac{1}{H} \sum_{h} \lambda_{h}$.
    It follows that $(\alpha, \beta, \bar{u}, \bar{\lambda}, 0)$ is feasible with objective value $\delta$, i.e., it is an optimal solution of the CGCP.
   
    Next, $(\alpha, \beta, \bar{u}, \bar{\lambda}, 0)$ is also an optimal solution of
    \begin{align*}
        \min_{u, \lambda} \ \ \ \ & \alpha^{T}\bar{x} - \beta\\
        s.t. \ \ \ 
        & \alpha = A^{T} u_{h} + \lambda_{h}, && \forall h,\\
        & \beta \leq b^{T} u_{h}, && \forall h,\\
        & \sum_{h} \knorm{\rho}{\lambda_{h}} \leq 1,\\
        & \lambda_{h} \in \K^{*}.
    \end{align*}
    Eliminating $\alpha, \beta$ yields
    \begin{align*}
        \min_{\lambda} \ \ \ \ & \lambda^{T} \bar{x}\\
        s.t. \ \ \ 
        &  \rho^{T} \lambda \leq \frac{1}{H},\\
        & \lambda \in \K^{*},
    \end{align*}
    whose dual, up to a change of sign in the objective value, writes
    \begin{align*}
        \min_{\eta} \ \ \ \ & \frac{1}{H} \eta\\
        s.t. \ \ \ 
        & \eta \rho \Kleq{\K} \bar{x},\\
        & \eta \geq 0,
    \end{align*}
    and has optimal value $H^{-1} \bar{\eta}$.
    Thus, $\delta = -H^{-1} \bar{\eta} = \bar{\lambda}^{T} \bar{x}$, which concludes the proof. \qed
\end{proof}

\begin{theorem}
    \label{thm:infeas:sufficient}
    If $\forall h, \bar{\eta} \geq H^{-1} \bar{\tau}_{h}$, then $(\bar{\lambda}, 0, 0, \bar{\lambda}, 0)$ is optimal for CGCP.
\end{theorem}
\begin{proof}
    Let $y_{h} = \frac{1}{H} \bar{x}$, $z_{h} = \frac{1}{H}$, and $\eta = \frac{\bar{\eta}}{H}$.
    
    We have $A y_{h} = z_{h} b, \forall h$.
    Then, $y_{h} + \eta \rho = \frac{1}{H}(\bar{x} + \bar{\eta}\rho) \in \K$.
    Finally, 
    \begin{align*}
        D_{h} y_{h} + \eta \sigma_{h}
        & = H^{-1} \left( D_{h} \bar{x} + \bar{\eta} \sigma_{h} \right)\\
        & \Kgeq{\Q_{h}} H^{-1} d_{h}\\
        & = z_{h} d_{h}.
    \end{align*}
    Thus, $(y, z, \eta)$ is feasible for the MCP and its objective value is $H^{-1}\bar{\eta}$, which concludes the proof.
    \qed
\end{proof}

Theorem \ref{thm:infeas:sufficient} shows that, if $\bar{x}$ is ``sufficiently" conic-infeasible, as measured by the magnitude of $\bar{\eta}$, then there exists a $\K^{*}$ cut that is an optimal solution of the CGCP.
Note that there is no guarantee that this optimal solution is unique --in general, it is not-- nor that all CGCP-optimal solutions are $\K^{*}$ cuts.
Nevertheless, we have observed that, whenever the condition of Theorem \ref{thm:infeas:sufficient} was met, the obtained solution was indeed a $\K^{*}$ cut.

This suggests several strategies to avoid generating $\K^{*}$ cuts when solving the CGCP.
First, one can check the value of $\bar{\eta}$ and, if large enough as per Theorem \ref{thm:infeas:sufficient}, avoid solving the CGCP and add an optimal $\K^{*}$ cut directly.
Nevertheless, our initial experiments suggest that only a small number of cases are captured by Theorem \ref{thm:infeas:sufficient}.
Second, one can increase the magnitude of $\rho$, thus reducing the value of $\bar{\eta}$, to the point where the assumptions of Theorem \ref{thm:infeas:sufficient} no longer hold.
Note that, as the magnitude of $\rho$ becomes arbitrarily large, the standard normalization becomes equivalent to the uniform normalization of Section \ref{sec:normalization:uniform}.
Third, instead of imposing a normalization condition, one could fix the cut violation to a positive value, e.g., $1$, and optimize a different objective; the feasibility of the CGCP is then guaranteed by the fact that $\bar{x}$ is conic-infeasible.
This approach directly relates to the \emph{reverse-polar} CGLP introduced in \cite{Serra2020_ReformulatingDisjunctiveCut}.
Finally, one can simply try to avoid conic-infeasible points in the first place, for instance by refining the outer-approximation before cuts are separated.

\section{Cut lifting and strenghtening}
\label{sec:LiftStrengthen}

    In this section, we present conic extensions of the classical lifting and strengthening procedures in MILP.
    For simplicity, we will assume that $\K =  \K_{1} \times \K_{2}$ where $\K_{i} \subseteq \R^{n_{i}}, i=1, 2$ and, correspondingly, we write
    \[
        A = \begin{bmatrix} A_{1} & A_{2} \end{bmatrix},
        \ D_{h} = \begin{bmatrix} D_{1, h} & D_{2, h} \end{bmatrix},
        \ \lambda_{h} = \begin{bmatrix} \lambda_{1, h}\\ \lambda_{2, h} \end{bmatrix},
        \ \alpha = \begin{bmatrix} \alpha_{1}\\ \alpha_{2} \end{bmatrix},
        \ \bar{x}= \begin{bmatrix} \bar{x}_{1}\\ \bar{x}_{2} \end{bmatrix}.
    \]
    Finally, we assume that $\bar{x}_{2} = 0$.

\subsection{Cut lifting}
\label{sec:lifting}
    
    In MILP, one can formulate the CGLP in a reduced space, by projecting out the null components of $\bar{x}$, then recover a valid cut in the original space by a lifting procedure.
    Here, we show that a similar technique can be used in the conic setting.
    
    Recall that $\bar{x}_{2} = 0$, and consider the reduced CGCP
    \begin{subequations}
    \label{eq:CGCP:red}
    \begin{align}
        \min \ \ \ & \alpha_{1}^{T}\bar{x}_{1} - \beta\\
        s.t. \ \ \ 
        & \alpha_{1} = A_{1}^{T}u_{h} + \lambda_{1, h} + D_{1, h}^{T}v_{h}, && \forall h,\\
        & \beta \leq b^{T}u_{h} - v_{h}^{T}d_{h}, && \forall h,\\
        & (u_{h}, \lambda_{1, h}, v_{h}) \in \R^{m} \times \K_{1}^{*} \times \Q_{h}^{*}, && \forall h.
    \end{align}
    \end{subequations}
    All the normalization conditions considered in Section \ref{sec:normalization} can be adapted to the reduced CGCP, namely by normalizing only the $\alpha_{1}, \lambda_{1}, v$ components as appropriate.
    For instance, the $\alpha$ normalization would write $\| \alpha_{1} \|_{*} \leq 1$, and the uniform normalization $\sum_{h} \knorm{\rho_{1}}{\lambda_{1, h}} \leq 1$, for $\rho_{1} \in \intr \K_{1}$.
    
    Any solution $(\alpha, \beta, u, \lambda, v)$ that is feasible for the CGCP yields a feasible solution $(\alpha_{1}, \beta, u, \lambda_{1, .}, v)$ for the reduced CGCP.
    In addition, since $\bar{x}_{2} = 0$, the corresponding objective values are the same.
    Reciprocally, a feasible solution for the CGCP can be obtained from a feasible solution of the reduced CGCP as shown in Lemma \ref{lem:lifting}.
    
    \begin{lemma}
        \label{lem:lifting}
        Let $(\alpha_{1}, \beta, u, \lambda_{1, .}, v)$ be feasible for the reduced CGCP \eqref{eq:CGCP:red}.
        Let
        \begin{align*}
            \alpha_{2} & \Kgeq{\K_{2}^{*}} A_{2}^{T}u_{h} + D_{h, 2}^{T}v_{h}, & \forall h,
        \end{align*}
        and, $\forall h$, let $\lambda_{2, h} = \alpha_{2} - A_{2}^{T}u_{h} - D_{2, h}^{T}v_{2}$.
        Then, $(\alpha, \beta, u, \lambda, v)$ is feasible for the CGCP \eqref{eq:CGCP} and $\alpha^{T}\bar{x} - \beta = \alpha_{1}^{T}\bar{x}_{1} - \beta$.
    \end{lemma}
    
    The proof of Lemma \ref{lem:lifting} is immediate.
    Note that the choice of $\alpha_{2}$ is not unique, especially if $\K_{2}^{*}$ possesses high-dimensional faces.
    Nevertheless, a reasonable requirement is to impose that $\alpha_{2}$ be minimal with respect to $\Kgeq{\K^{*}}$, i.e., that there does not exist a valid $\tilde{\alpha}_{2} \preceq_{\K^{*}} \alpha_{2}$.
    Indeed, if $\alpha_{2}$ is not $\Kgeq{\K^{*}}$-minimal, then $\alpha^{T}x \geq \beta$ is trivially dominated by $\tilde{\alpha}^{T}x \geq \beta$.
    
    A $\Kgeq{\K^{*}}$-minimal $\alpha_{2}$ is obtained by solving the lifting conic problem (LCP)
    \begin{subequations}
    \label{eq:LCP}
    \begin{align}
        \label{eq:LCP:objective}
        (LCP) \ \ \ \min_{\alpha_{2}} \ \ \ & \rho_{2}^{T}\alpha_{2}\\
        s.t. \ \ \ 
        \label{eq:LCP:lift}
        & \alpha_{2} \Kgeq{\K_{2}^{*}} A_{2}^{T}u_{h} + D_{h, 2}^{T}v_{h}, & \forall h,
    \end{align}
    \end{subequations}
    where $\rho_{2} \in \intr \K_{2}$, and denote by $\lambda_{2, h} \in \K_{2}^{*}$ the (conic) slack associated to \eqref{eq:LCP:lift}.
    First, since $\K_{2}$ is a proper cone, the LCP is strongly feasible.
    Second, we have
    \begin{align}
    \label{eq:LCP:obj:lambda}
        \rho_{2}^{T}\alpha_{2} &= 
        \rho_{2}^{T}\lambda_{2, h}
        + \rho_{2}^{T}\left( A_{2}^{T}u_{h} + D_{h, 2}^{T}v_{h} \right),
        & \forall h,
    \end{align}
    thereby showing that the objective value of the LCP is bounded below.
    In addition, let $(\alpha_{2}, \lambda_{2, .})$ be a feasible solution for the LCP with objective value $Z$.
    It then follows that
    \begin{align}
    \label{eq:LCP bound lambda}
        \forall h, \ \ \ \knorm{\rho_{2}}{\tilde{\lambda}_{2, h}}
        & \leq Z - \rho_{2}^{T}\left( A_{2}^{T}u_{h} + D_{h, 2}^{T}v_{h} \right)
    \end{align}
    in any feasible solution $(\tilde{\alpha}_{2}, \tilde{\lambda}_{2, .})$ with objective value $\tilde{Z} \leq Z$.
    Equation \eqref{eq:LCP bound lambda} implicitly bounds the magnitude of $\lambda_{2, .}$, thereby ensuring that the LCP is solvable.
    Finally, if $\alpha_{2}$ and $\tilde{\alpha}_{2}$ are feasible and $\tilde{\alpha}_{2} \preceq_{\K^{*}} \alpha_{2}$, then $\rho_{2}^{T}\tilde{\alpha}_{2} < \rho_{2}^{T}\alpha_{2}$ and $\alpha_{2}$ cannot be an optimal solution.
    Thus, any optimal solution of the LCP is $\Kgeq{\K^{*}}$-minimal.
    
    Whenever $\K_{2}$ is not irreducible, the LCP can be decomposed per conic component, yielding smaller problems.
    In the linear case, taking $\K_{2} = \R_{+}$, the LCP reduces to
    \begin{align*}
        \alpha_{2} = \max_{h} \left\{ A_{2}^{T}u_{h} + D_{2, h}^{T}v_{h} \right\},
    \end{align*}
    which is the classical lifting procedure for disjunctive cuts in MILP \cite{Fischetti2011}.
    
    Lemma \ref{lem:lifting} does not account for the normalization constraint in the CGCP.
    Although the lifting procedure does not modify the objective value, in general, the lifted solution $(\alpha, \beta, u, \lambda, v)$ is not an optimal solution of the normalized original CGCP.
    Thus, the reduction in the size of the CGCP, and the associated computational gains, come at the expense of potentially weaker cuts.

\subsection{Cut strengthening}
\label{sec:strengthening}

    Balas and Jeroslow's original derivation of monoidal strengthening for disjunctive cuts  \cite{BalasJeroslow1980_StrengtheningCutsMixed} exploited the non-negativity and integrality of \emph{individual} variables to strengthen the corresponding cut coefficients.
    Since, in general, a conic constraint may involve several variables, it is not clear whether and how one can extend this approach to the conic setting.
    Thus, we restrict our attention to split cuts, and propose a conic extension of monoidal strengthening that builds on the geometric idea of Wolsey's proof of Theorem 2.2 in \cite{BalasEtAl1996_Mixed01}.
    
    For simplicity, we consider the pure integer case.
    Given a split disjunction $(\pi^{T}x \leq \pi_{0}) \vee (\pi^{T}x \geq \pi_{0} +  1)$, the CGCP writes
    \begin{align*}
        \min \ \ \ 
        & \alpha^{T}\bar{x} - \beta\\
        s.t. \ \ \ 
        & \alpha = A^{T}u_{1} + \lambda_{1} - v_{1} \pi\\
        & \alpha = A^{T}u_{2} + \lambda_{2} + v_{2} \pi\\
        & \beta \leq b^{T}u_{1} - v_{1} \pi_{0}\\
        & \beta \leq b^{T}u_{2} + v_{2} (\pi_{0} + 1)\\
        & \lambda_{1}, \lambda_{2} \in \K^{*}, v_{1}, v_{2} \geq 0.
    \end{align*}
    
    \begin{lemma}
    \label{lem:strengthening}
        Let $(\alpha, \beta, u, \lambda, v)$ be feasible for the CGCP,
        and let $\tilde{\alpha}_{2}$ and $\delta \pi \in \Z^{n_{2}}$ such that
        \begin{align*}
            \tilde{\alpha}_{2} & \Kgeq{\K_{2}^{*}} A_{2}^{T}u_{1} - v_{1} (\pi_{2} + \delta \pi),\\
            \tilde{\alpha}_{2} & \Kgeq{\K_{2}^{*}} A_{2}^{T}u_{2} + v_{2} (\pi_{2} + \delta \pi).
        \end{align*}
        Then, $\tilde{\alpha}^{T}x \geq \beta$ is a valid inequality for the disjunctive set
        \begin{align*}
            \tilde{\D} =
            \left\{
                x
                \ \middle| \ 
                \begin{array}{l}
                    A x = b, x \in \K\\
                    \tilde{\pi}^{T}x \leq \pi_{0}
                \end{array}
            \right\}
            \cup
            \left\{
                x
                \ \middle| \ 
                \begin{array}{l}
                    A x = b, x \in \K\\
                    \tilde{\pi}^{T}x \geq \pi_{0} + 1
                \end{array}
            \right\}
            ,
        \end{align*}
        where $\tilde{\alpha} = (\alpha_{1}, \tilde{\alpha}_{2})$ and $\tilde{\pi} = (\pi_{1}, \pi_{2} + \delta \pi)$.
    \end{lemma}
    
    In the mixed-integer case, one simply needs to set to zero, in Lemma \ref{lem:strengthening}, the components of $\delta \pi$ that correspond to continuous variables.
    Since $\delta \pi \in \Z^{n_{2}}$, $\tilde{\pi} \in \Z^{n}$ and $(\tilde{\pi}^{T}x \leq \pi_{0}) \vee (\tilde{\pi}^{T}x \geq \pi_{0} + 1) $ is a valid disjunction for $\X$.
    Thus, the strengthened inequality is valid for $\X$.
    
    Similar to the lifting case, a reasonable requirement is for $\tilde{\alpha}_{2}$ to be minimal with respect to $\Kgeq{\K_{2}^{*}}$.
    Following the same approach as in Section \ref{sec:lifting}, we obtain the \emph{Cut-Strengthening Problem}
    \begin{subequations}
    \label{eq:CSP}
        \begin{align}
        \label{eq:CSP:objective}
        (CSP) \ \ \ \min_{\tilde{\alpha}_{2}} \ \ \ & \rho_{2}^{T}\tilde{\alpha}_{2}\\
        s.t. \ \ \ 
            & \tilde{\alpha}_{2} \Kgeq{\K_{2}^{*}} A_{2}^{T}u_{1} - v_{1} (\pi_{2} + \delta \pi),\\
            & \tilde{\alpha}_{2} \Kgeq{\K_{2}^{*}} A_{2}^{T}u_{2} + v_{2} (\pi_{2} + \delta \pi),\\
            & \delta \pi \in \Z^{n_{2}}.
    \end{align}
    \end{subequations}
    
    Similar to the LCP, the CSP can be decomposed per conic component, so that its resolution remains tractable.
    Furthermore, in the linear case with $\K_{2} = \R_{+}$, the CSP reduces to
    \begin{align*}
        \tilde{\alpha}_{2} 
        &= \max_{\delta \in \Z} \left\{ \min \left( A_{2}^{T}u_{1} - v_{1}\delta, A_{2}^{T}u_{2} + v_{2} \delta \right) \right\},
    \end{align*}
    which is the classical monoidal strengthening of Balas and Jeroslow \cite{BalasJeroslow1980_StrengtheningCutsMixed}.

\section{Computational results}
\label{sec:results}

    In this section, we investigate the practical behavior of the normalization conditions of Section \ref{sec:normalization} along two lines: the progression of the gap closed, and the characteristics of the obtained cuts.
    Indeed, the choice of normalization impacts the numerical stability and computational efficiency of solving the CGCP, thereby affecting the rate at which gap is closed, in terms of both time and number of rounds.
    
    Several solvers, e.g., CPLEX, Gurobi and Mosek, support classes of \MICONIC\ problems.
    Nevertheless, to the best of our knowledge, CPLEX is the only one that exploits non-linear information when generating lift-and-project cuts.
    Thus, we use CPLEX as a baseline, and restrict our comparison to MISOCP instances, which is the only class of non-linear \MICONIC\ problems supported by CPLEX.

\subsection{Instances}
    
    We select an initial testset of $114$ MISOCP instances from the CBLIB \cite{CBLIB2014} collection.
    Each instance is first reformulated in the standard form \eqref{eq:MICP} and, since the MathOptInterface wrapper of CPLEX does not directly support constraints of the form $x \in \K$ where $\K$ is a rotated second-order cone, all such constraints are reformulated as second-order cone constraints.

    Each instance is solved using the CPLEX outer-approximation algorithm on a single thread, and all other parameters are left to their default value.
    Instances with a root gap smaller than $1 \%$ are removed from the testset, as well as those for which no integer-feasible solution was found by CPLEX after one hour of computing time.
    This yields a testset of $101$ instances, divided into 7 groups.
    Table \ref{tab:res:instances} reports, for each group, the number of instances (\#Inst) in that group, and the average number of: variables, integer variables, constraints, non-polyhedral cones, and non-zero coefficients.

    \begin{table}
        \centering
        \caption{Instance statistics}
        \label{tab:res:instances}
        \begin{tabular}{lrrrrrr}
            \toprule
            Group              & \#Inst &     Variables &     Integers &     Constraints &     Cones &      Nz coeff.\\
            \midrule
            \texttt{clay}         & 12 &    829 &     35 &    659 &     80 &   1711\\
            \texttt{flay}         & 10 &    535 &     28 &    408 &      4 &   1121\\
            \texttt{fmo}          & 39 &    620 &     48 &    479 &     15 &   1843\\
            \texttt{slay}         & 14 &   1217 &     92 &    936 &     14 &   2653\\
            \texttt{sssd}         & 16 &    694 &    153 &    547 &     18 &   1414\\
            \texttt{tls}          &  2 &    295 &     61 &    227 &     10 &    817\\
            \texttt{uflquad}      &  8 &  37569 &     23 &  30203 &   3750 &  71341\\
            \bottomrule
        \end{tabular}
    \end{table}

\subsection{Implementation details}

    Our implementation\footnote{Code available at \url{https://github.com/mtanneau/CLaP}} is coded in Julia.
    All solvers, namely, CPLEX 12.10 \cite{CPLEX}, Gurobi 9.0 \cite{Gurobi} and Mosek 9.2 \cite{Mosek}, are accessed through the solver-agnostic interface \texttt{MathOptInterface} \cite{Legat2020_MOI}.
    Experiments are carried out on an Intel Xeon E5-2637@3.50GHz CPU, 128GB RAM machine running Linux.

    \paragraph{Baseline.}
        We use CPLEX internal lift-and-project cut generation at the root node as a baseline.

        For each instance, we run CPLEX outer-approximation algorithm on a single thread, no presolve, no heuristics, and all cuts de-activated with the exception of lift-and-project cuts which are set to the most aggressive setting.
        The \texttt{CutsFactor} parameter is set to $10^{30}$, thereby removing any limit on the number of cuts that can be added to the formulation, and the maximum number of cutting plane passes is set to $200$.
        Finally, the node limit is set to zero, i.e., only the root node is explored, and we set a time limit of $1$ hour.
        
        Each CPLEX ``cut pass" consists\footnote{Personal communication with CPLEX developers.} of one round of cuts, plus additional components such as heuristics and reduced cost fixing.
        Cut generation stops as soon as no violated cuts are found, or all violated cuts are rejected.
        Since we deactivate heuristics and presolve, it is fair to assume that, in the present setting, one CPLEX pass corresponds to one round of cuts.
        Nevertheless, due to CPLEX internal work limits, e.g., numerical tolerances on a cut's coefficients, early termination may occur even though violated cuts were identified.
        Whether or not this is the case cannot be inferred from the CPLEX log.

    \paragraph{Cut separation.}
        Cuts are separated by rounds in a callback, and submitted to CPLEX as user cuts.
        Each round proceeds as follows.
        
        First, $\bar{x}$ is cleaned, i.e, we set $\bar{x}_{j} = 0$ for all $j$ such that $| \bar{x}_{j} | \leq 10^{-7}$, and $\bar{\eta}$ is computed as per Equations \eqref{eq:infeasible:etabar}.
        If $\bar{\eta} > \epsilon_{\K}$, the current  outer-approximation is refined by adding violated $\K^{*}$ cuts.
        This refinement step is cheap, and it is repeated until $\bar{x} \leq \epsilon_{\K}$, up to a maximum of $50$ times.
        Large values for $\epsilon_{\K}$ increase the likelihood of generating $\K^{*}$ cuts from the CGCP, while small values can lead to an adverse tailing-off effect within the refinement process; following initial tests, we set $\epsilon_{\K} = 0.05$.
        
        Second, for each fractional coordinate of $\bar{x}$, the CGCP is formed and solved to generate a lift-and-project cut.
        We do not project the CGCP onto the support of $\bar{x}$, i.e., no lifting is performed.
        For the standard normalization, the sufficient conditions of Theorem \ref{thm:infeas:sufficient} are checked first and, if met, $\K^{*}$ cuts are added and no CGCP is solved.
        If no CGCP solution is available, or if the objective value of the CGCP is greater than $-10^{-4}$, no cut is generated and we proceed to the next fractional coordinate.
        
        Third, the values of $v_{1}$ and $v_{2}$ are checked to identify $\K^{*}$ cuts, which are disaggregated and added to the formulation.
        Otherwise, the cut is strengthened following the procedure of Section \ref{sec:strengthening}.
        For stability, strengthening is performed only if $|v_{1}| + |v_{2}| \geq 10^{-4}$.
        Then, also for numerical stability, we set $\alpha_{j}$ to zero for every $j$ such that $|\alpha_{j}| \leq 10^{-7}$.
        Similarly, if $|\beta| \leq 10^{-8}$, the cut's right-hand side is set to zero.
        Finally, the cut is passed to the solver.

    \paragraph{Compact CGCP.}
        Lift-and-project cuts are obtained from elementary split disjunctions, i.e., disjunctions of the form
        \begin{align*}
            \left(
                -\pi^{T}x \geq \pi_{0}
            \right)
            \vee
            \left(
                \pi^{T}x \geq \pi_{0} +1
            \right),
        \end{align*}
        where $\pi_{0} \in \Z$ and $\pi = e_{j}$ for given $j \in \{1, ..., p \}$.
        When formulating the CGCP, we set $u_{1}$ to zero, and substitute out $\alpha$ and $\beta$.
        This yields the compact CGCP
        \begin{align*}
            \min_{u, \lambda, v} \ \ \ 
            & \bar{x}^{T}\lambda_{1} - (\pi^{T}\bar{x} - \pi_{0})v_{1} + t_{1}\\
            s.t. \ \ \ 
            & A^{T}u_{2} + (\lambda_{2} - \lambda_{1}) + (v_{1} + v_{2}) \pi = 0\\
            & b^{T}u_{2} + (t_{1} - t_{2}) + (v_{1} + v_{2}) \pi_{0} + v_{2} = 0\\
            & \lambda_{1}, \lambda_{2} \in \K^{*},\\
            & v_{1}, v_{2} \geq 0\\
            & t_{1}, t_{2} \geq 0,
        \end{align*}
        where $t_{1}, t_{2}$ are non-negative slacks associated with constraints \eqref{eq:CGCP:f2}.
        The relation $\alpha = \lambda_{1} - v_{1} \pi$ allows to formulate the normalization condition in the $\lambda, v$ space.
        
        In practice, the equality constraints $A \bar{x} = b$ are satisfied only up to numerical tolerances.
        This can cause the CGCP to be unbounded whenever the $u$ multipliers are not bounded.
        We have observed that setting $u_{1}$ to zero, and writing the objective as in the compact CGCP above, greatly improve the numerical stability of the CGCP, especially when $\bar{x}$ is obtained from an interior-point method.

    \paragraph{Other details.}
        Cuts that are generated are never added to the CGCP formulation, i.e., we only separate rank-1 lift-and-project cuts.
        All CGCPs are solved with Gurobi, which we found to be more robust here.
        We use the barrier algorithm with a single thread, no presolve, and we disable the computation of dual variables by setting the \texttt{QCPDual} parameter to 0.

        We implement a simple procedure to identify implied-integer variables.
        If a variable $y$ appears in a constraint of the form $a^{T}x \pm y = b$, where $b \in \Z$, all coefficients of $a$ are integers, and $x$ is a vector of integer or implied-integer variables, then $y$ is an implied-integer variable.
        This step is repeated until no additional implied-integer variable is detected.
        Then, when generating cuts, the coefficients of implied-integer variables can be strengthened using the technique of Section \ref{sec:strengthening}.
        
        Since we use an outer-approximation algorithm, we do not include the trivial normalization in our experiments.
        The $\alpha$ normalization is formulated using the $\ell_{2}$ norm.
        For the polar normalization, we take $\gamma = x^{*} - \bar{x}$, where $x^{*}$ is the analytic center of $\C$, obtained from solving $(CP)$ with zero objective by an interior-point method.
        Finally, for the standard and uniform normalization, and the computation of $\bar{\eta}$ in Equation \eqref{eq:infeasible:etabar}, we set $\rho = (\rho_{1}, ..., \rho_{N})$, where
        \begin{align*}
            \rho_{i} &= \left\{
                \begin{array}{cl}
                    1 & \text{ if } \K_{i} = \R_{+}\\
                    -1 & \text{ if } \K_{i} = \R_{-}\\
                    (1, 0, 0) & \text{ if } \K_{i} = \mathcal{L}_{3}
                \end{array}
            \right.
            ,
        \end{align*}
        and, since we only consider split cuts, we take $\sigma_{1} = \sigma_{2} = 1$.

\subsection{Gap closed}
    
    Tables \ref{tab:res:gap:10rounds} and \ref{tab:res:gap:200rounds} report the performance of each approach after 10 and 200 rounds of cuts, respectively.
    For each normalization (Alpha, Polar, Standard, Uniform) and the baseline (CPLEX), we report the geometric mean of the percent gap closed (Gap), and of the computing time (CPU), in seconds.
    The percent gap closed is measured as
    \begin{align*}
        Gap = 100 \times \frac{Z^{*} - Z_{CP}}{Z_{MICP} - Z_{CP}},
    \end{align*}
    where $Z_{CP}, Z_{MICP}, Z^{*}$ are the optimal value of the continuous relaxation, the objective value of the best known integer solution, and the current lower bound after adding cuts, respectively.
    All geometric means are computed with a shift of $1$.
    
    Recall that, in our approach, each round may include the separation of $\K^{*}$ cuts to refine the current outer approximation while, to the best of our knowledge, a CPLEX pass does not.
    In addition, total computing times include the time of building the CGCP, which typically represent $30$ to $40\%$ of total time, and up to $90\%$ for \texttt{uflquad} instances.
    These large building times are in part due to limitations of the Gurobi Julia wrapper, and could be significantly reduced by using a lower-level interface.
    Thus, direct comparisons between CGCP-based approaches and the baseline should be cautious.
    
    \begin{table}
        \centering
        \caption{Gap closed and computing time - 10 rounds}
        \label{tab:res:gap:10rounds}
        \begin{tabular}{lrrrrrrrrrr}
            \toprule
            &  \multicolumn{2}{c}{CPLEX} & \multicolumn{2}{c}{Alpha} & \multicolumn{2}{c}{Polar} & \multicolumn{2}{c}{Standard} & \multicolumn{2}{c}{Uniform}\\
            \cmidrule(rl){2-3} \cmidrule(rl){4-5} \cmidrule(rl){6-7} \cmidrule(rl){8-9} \cmidrule(rl){10-11}
            Group  & Gap & CPU & Gap & CPU & Gap & CPU & Gap & CPU & Gap & CPU \\
            \midrule
\texttt{clay}      &  2.13 &    1.0  &  0.00 &   12.0  &  0.00 &    4.4  &  5.69 &    6.0  &  2.92 &    5.7 \\
\texttt{flay}      &  0.32 &    0.3  &  0.45 &    3.7  &  1.50 &    2.1  &  0.88 &    2.1  &  2.40 &    1.7 \\
\texttt{slay}      &  3.67 &    2.4  &  0.65 &   23.2  &  0.15 &    8.8  & 15.81 &   11.2  & 10.57 &    8.7 \\
\texttt{fmo}       &  6.42 &    0.7  &  0.00 &   11.7  &  0.00 &    5.2  & 25.31 &    5.6  & 25.59 &    4.6 \\
\texttt{sssd}      & 45.07 &    0.9  & 36.49 &   12.3  &  0.00 &    6.6  & 93.44 &    5.3  & 96.47 &    7.9 \\
\texttt{tls}       &  8.13 &    0.5  &  2.56 &    2.6  & 20.57 &    1.9  &  8.47 &    1.4  & 13.36 &    1.5 \\
\texttt{uflquad}   &  0.46 &   56.8  &  0.00 & 1307.9  & 12.50 & 1568.7  & 18.61 & 1769.8  & 18.98 & 1721.5 \\
\midrule
All                &  5.24 &    1.5  &  1.02 &   16.9  &  0.46 &    8.5  & 17.97 &    9.2  & 17.27 &    8.6 \\
            \bottomrule
        \end{tabular}
    \end{table}
    
    \begin{table}
        \centering
        \caption{Gap closed and computing time - 200 rounds}
        \label{tab:res:gap:200rounds}
        \begin{tabular}{lrrrrrrrrrrrrrrr}
            \toprule
            &  \multicolumn{2}{c}{CPLEX} & \multicolumn{2}{c}{Alpha} & \multicolumn{2}{c}{Polar} & \multicolumn{2}{c}{Standard} & \multicolumn{2}{c}{Uniform}\\
            \cmidrule(rl){2-3} \cmidrule(rl){4-5} \cmidrule(rl){6-7} \cmidrule(rl){8-9} \cmidrule(rl){10-11}
            Group  & Gap & CPU & Gap & CPU & Gap & CPU & Gap & CPU & Gap & CPU \\
            \midrule
\texttt{clay}      & 24.48 &   32.3  &  0.00 &  251.4  &  0.00 &    6.2  & 22.67 &  184.9  & 22.72 &  165.7 \\
\texttt{flay}      &  6.42 &    1.0  &  2.62 &   50.1  &  8.59 &    8.6  &  6.65 &   23.0  & 41.25 &   19.1 \\
\texttt{slay}      & 24.70 &    9.7  & 16.75 &  490.2  &  0.20 &   19.6  & 83.44 &  196.2  & 84.85 &   87.7 \\
\texttt{fmo}       & 25.89 &   10.5  &  0.00 &  243.0  &  0.00 &   21.7  & 26.00 &   68.1  & 26.01 &   46.7 \\
\texttt{sssd}      & 95.88 &    4.1  & 95.65 &  284.4  &  0.00 &    8.3  & 99.43 &   14.9  & 99.46 &   17.8 \\
\texttt{tls}       & 41.47 &   12.8  &  3.07 &   19.9  & 20.96 &    7.1  & 31.90 &   13.0  & 34.76 &   14.5 \\
\texttt{uflquad}   &  5.39 &  475.0  &  0.00 & 1307.9  & 16.22 & 3940.6  & 23.15 & 4178.0  & 23.81 & 4140.0 \\
\midrule
All                & 24.79 &   11.9  &  2.59 &  256.9  &  0.71 &   21.9  & 32.67 &   85.0  & 39.13 &   65.5 \\
            \bottomrule
        \end{tabular}
    \end{table}

    First, overall, with the exception of the \texttt{tls} and \texttt{uflquad} instances, the polar normalization performs the worst, with $0.46 \%$ and $0.71 \%$ gap closed after up to $10$ and $200$ rounds, respectively.
    This poor performance is primarily caused by premature termination of the cut generation procedure: in numerous cases, the CGCP is unbounded, no solution is available, and no cut is generated.
    
    Second, the $\alpha$ normalization is the slowest on average, with computing times that are typically $2$ to $3$ times higher than other approaches.
    This behavior results from the CGCP \eqref{eq:CGCP:alpha} being more computationally intensive than for alternative normalizations.
    Indeed, the normalization \eqref{eq:CGCP:alpha:normalization} is non-linear and, as highlighted in Section \ref{sec:normalization:alpha}, the CGCP \eqref{eq:CGCP:alpha} may not be solvable, then causing slow convergence and more interior-point iterations.
    The $\alpha$ normalization also performs second-worst with respect to gap closed, with $1.02 \%$ and $2.59 \%$ gap closed after $10$ and $200$ rounds, respectively.
    
    Third, the standard and uniform normalizations display similar performance, with the exception of \texttt{flay} instances, for which the uniform normalization closes significantly more gap than all other approaches.
    While both normalizations close similar gaps, namely, $17.97 \%$ and $17.27 \%$ after $10$ rounds, and $32.67 \%$ and $39.13 \%$ after 200 rounds,
    the uniform normalization is slightly faster overall, taking an average $65.5$s to complete up to 200 rounds, against $85.0$s for the standard normalization.
    Furthermore, no numerical issues were recorded for either normalization, thereby demonstrating the benefits of ensuring strong feasibility of both the CGCP and the MCP.
    
    Fourth and last, the standard and uniform normalizations are competitive with CPLEX in terms of gap closed.
    Both normalizations close more gap than CPLEX after 200 rounds, mainly on \texttt{slay} and \texttt{uflquad} instances, as well as \texttt{flay} instances for the uniform normalization.
    In particular, they close over three times more gap than CPLEX in the first 10 rounds.
    Despite the limited validity of the comparison, this behavior is encouraging, since closing more gap early is a desirable feature in practice.

    Next, Figures \ref{fig:res:gap:flay02m}, \ref{fig:res:gap:sssd} and \ref{fig:res:gap:tls4} illustrate the progression of the gap closed for instances \texttt{flay02m}, \texttt{sssd-weak-15-4} and \texttt{tls4}, respectively.
    Each figure displays, for each normalization and the baseline, the percent gap closed as a function of the number of rounds, and of computing time.
    Again, recall that a direct comparison between our approach and CPLEX is not meaningful on a per-round basis, thus, we focus on overall trends.

    \begin{figure}
        \centering
            \includegraphics[width=100mm]{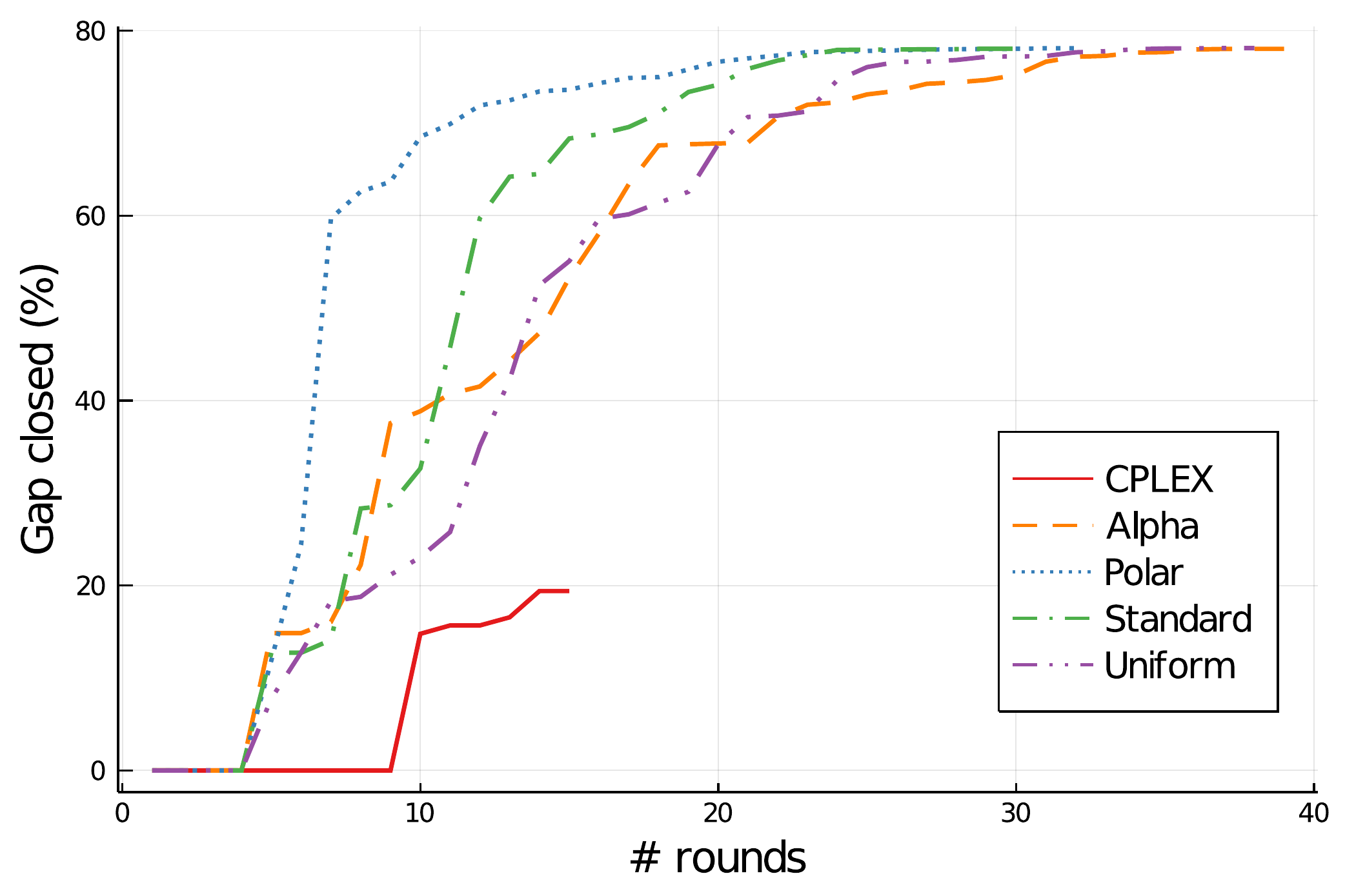}
        \hfill
            \includegraphics[width=100mm]{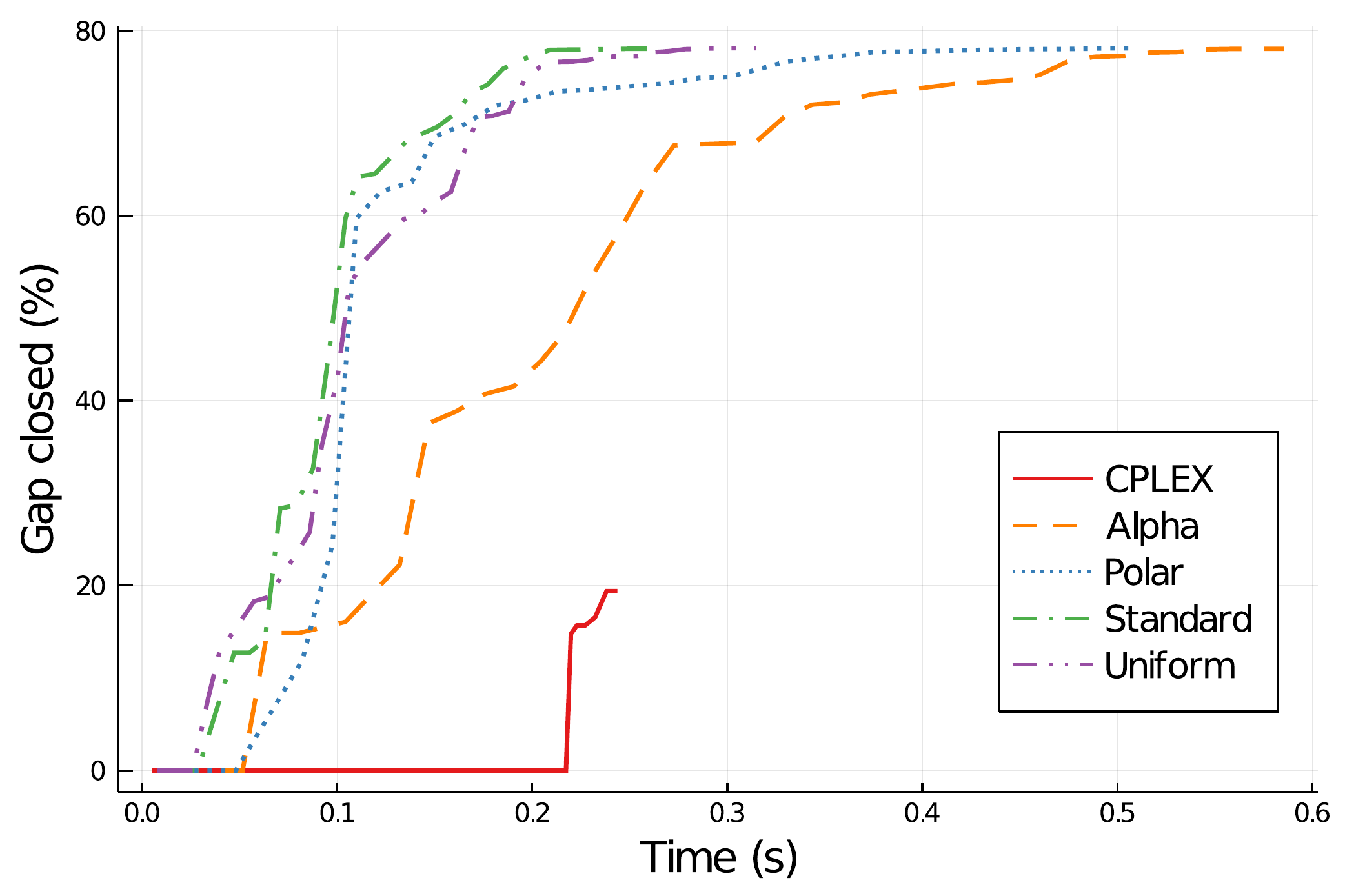}
        \caption{Instance \texttt{flay02m}: gap closed per round (top) and time (bottom)}
        \label{fig:res:gap:flay02m}
    \end{figure}
    
    \begin{figure}
        \centering
            \includegraphics[width=100mm]{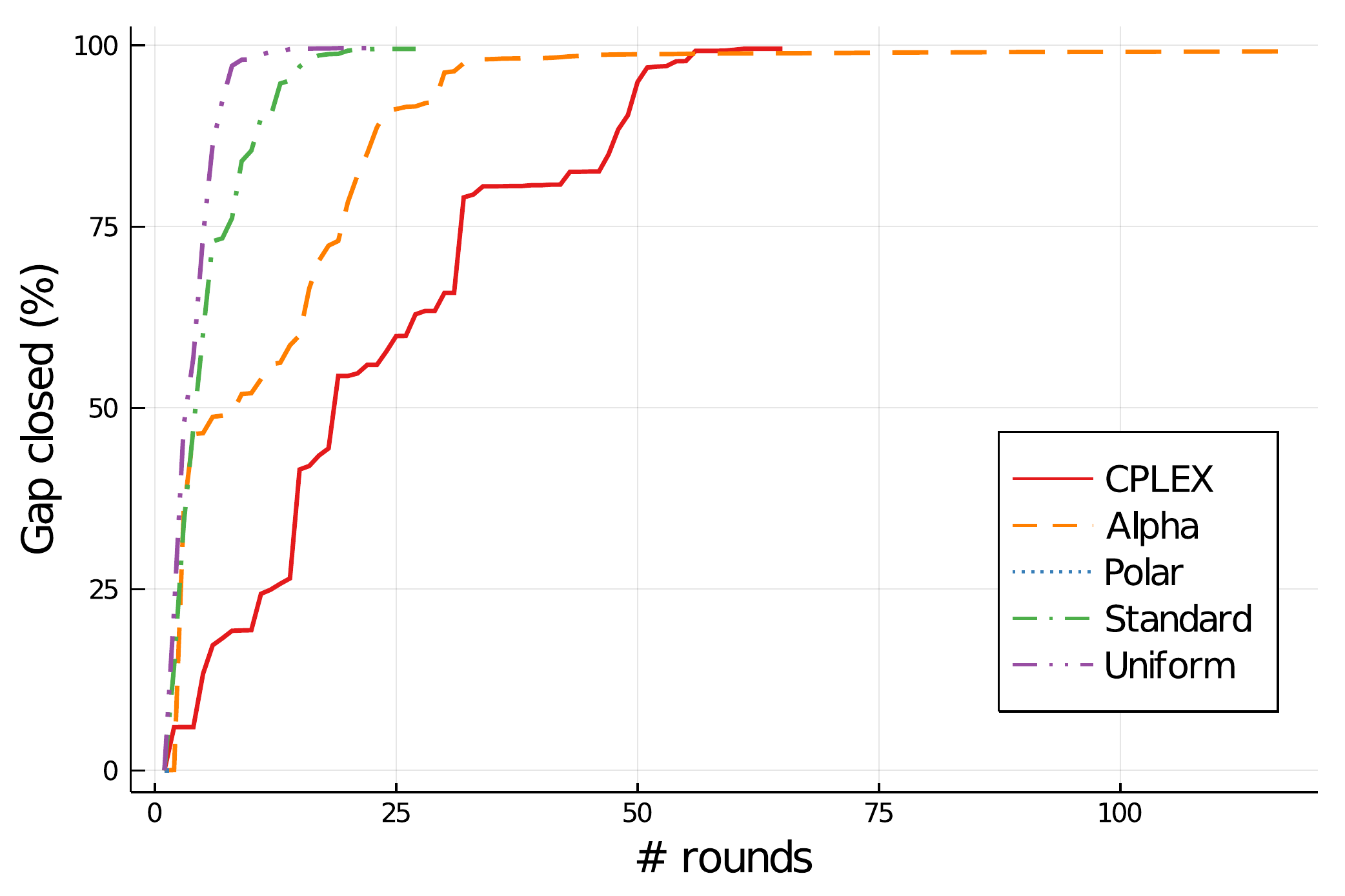}
        \hfill
            \includegraphics[width=100mm]{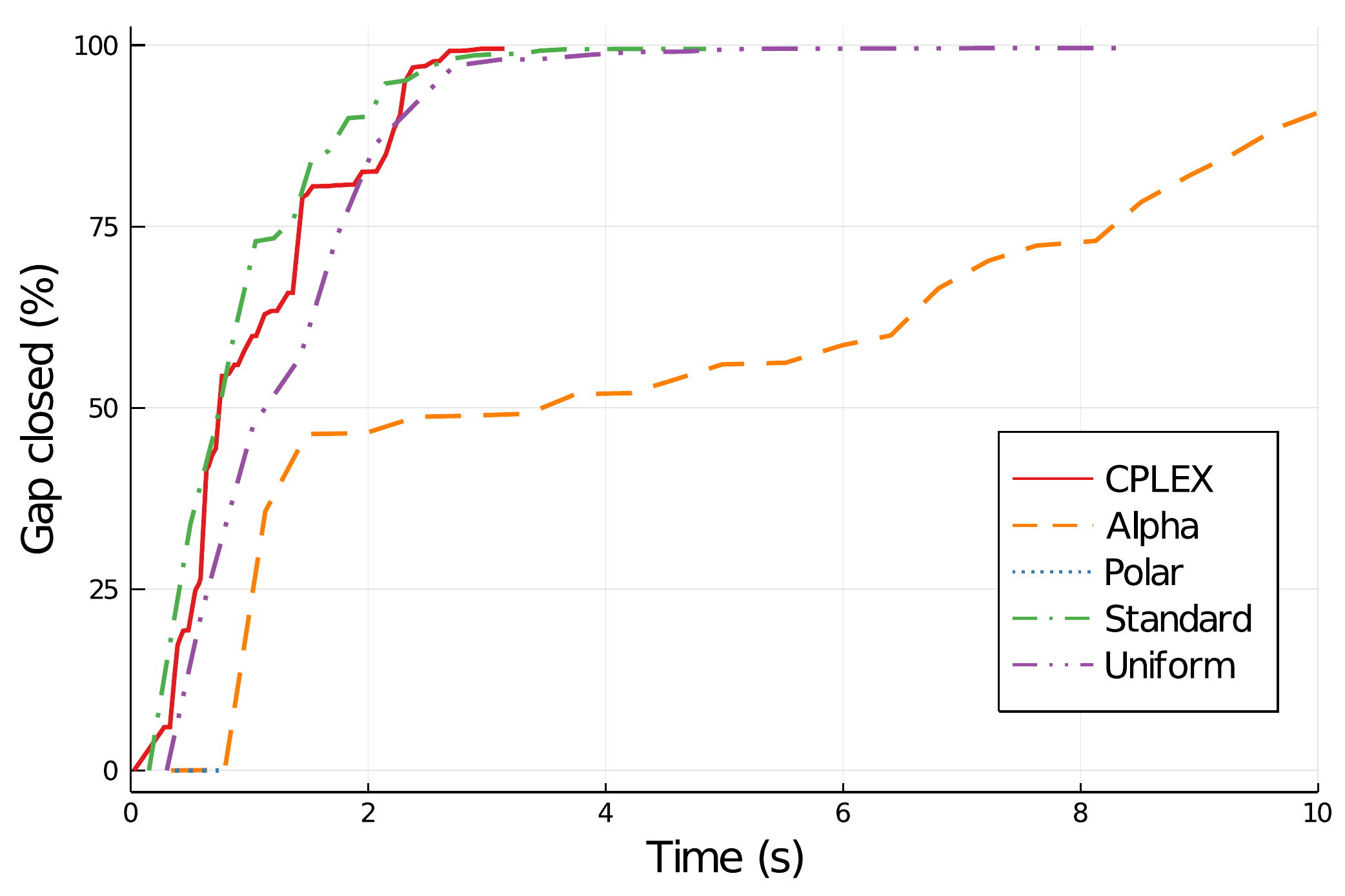}
        \caption{Instance \texttt{sssd-weak-15-4}: gap closed per round (top) and time (bottom). The polar normalization was terminated after two rounds.}
        \label{fig:res:gap:sssd}
    \end{figure}

    \begin{figure}
        \centering
            \includegraphics[width=100mm]{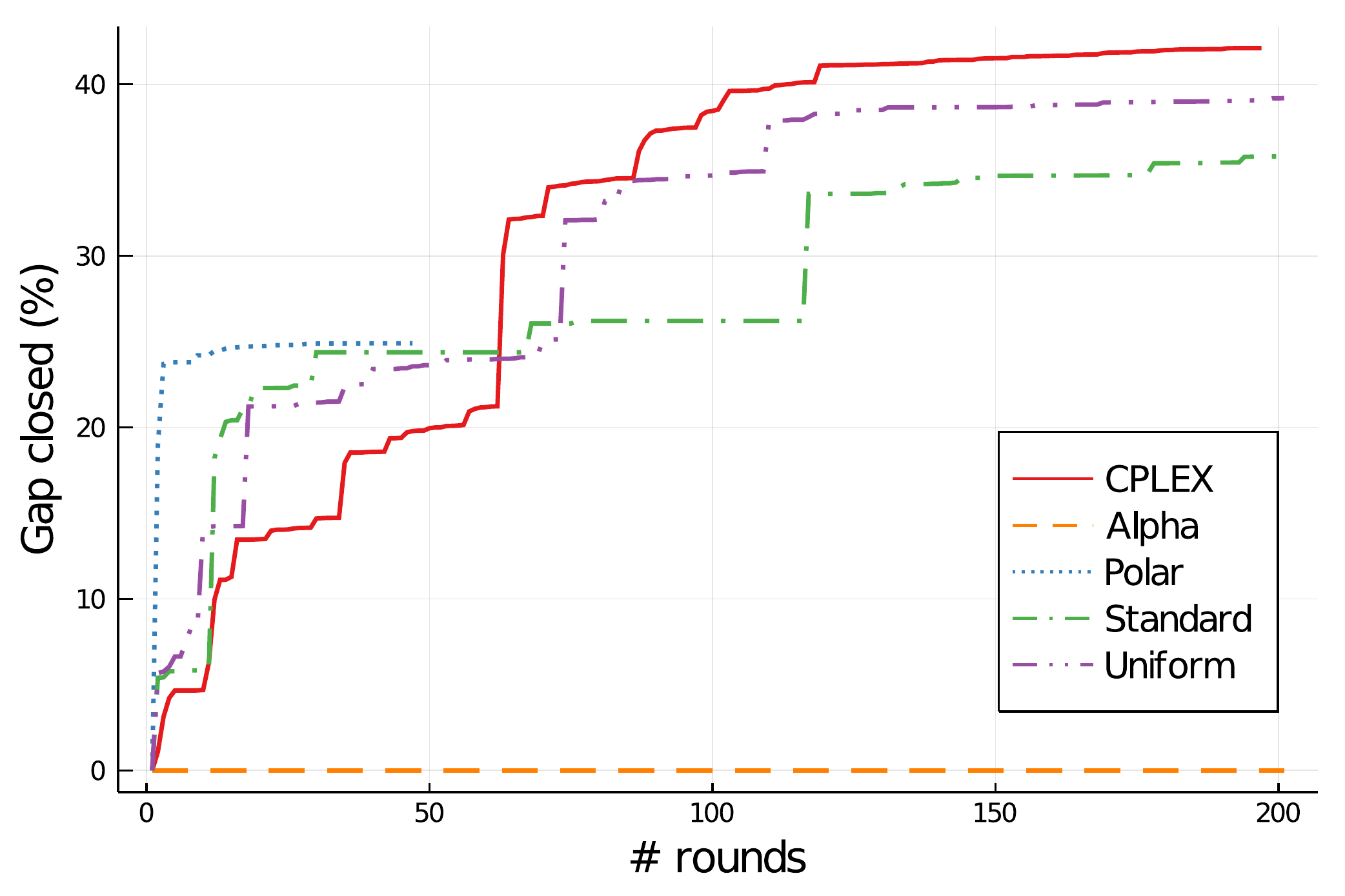}
        \hfill
            \includegraphics[width=100mm]{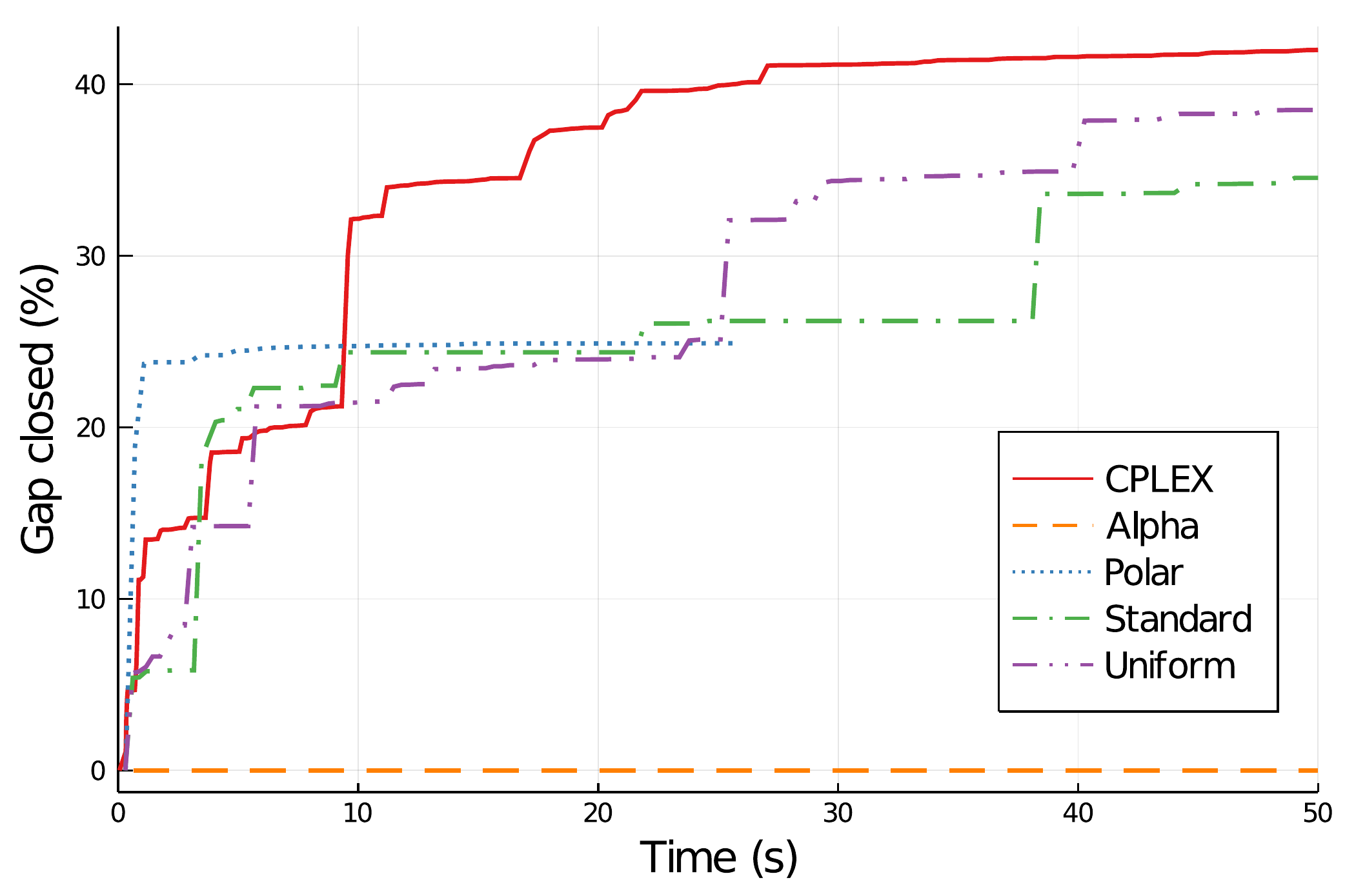}
        \caption{Instance \texttt{tls4}: gap closed per round (top) and time (bottom)}
        \label{fig:res:gap:tls4}
    \end{figure}

    Overall, Figures \ref{fig:res:gap:flay02m}, \ref{fig:res:gap:sssd} and \ref{fig:res:gap:tls4} corroborate the previous observations from Tables \ref{tab:res:gap:10rounds} and \ref{tab:res:gap:200rounds}: the $\alpha$ normalization is the slowest, while the standard and uniform normalizations display similar progressions and allow to close more gap than CPLEX in the first few rounds.
    We also observe in Figure \ref{fig:res:gap:flay02m} and Figure \ref{fig:res:gap:tls4} that, when the polar normalization does not run into numerical issues, it can outperform other approaches for the most gap closed in the first few rounds of cuts.
    
    Furthermore, when comparing the gap closed with respect to time, we observe that the CGCP-based separation procedure is competitive with CPLEX.
    This is most striking for instance \texttt{flay02m}, where all four normalizations outperform CPLEX.
    Note that CPLEX terminates after applying only $15$ rounds of cuts, most likely due to internal work limits.
    For instances \texttt{sssd-weak-15-4} and \texttt{tls4}, the standard and uniform normalizations are initially on-par with CPLEX, while the $\alpha$ normalization is the slowest.

\subsection{Cut sparsity}

    We now study, for each normalization, the characteristics of the cuts obtained from solving the CGCP.
    Relevant statistics are reported after 10 and 200 rounds, in Table \ref{tab:res:sparstity:10rounds} and Table \ref{tab:res:sparsity:200rounds}, respectively.
    For each normalization, we report the geometric mean of the total number of disaggregated $\K^{*}$ cuts ($\K^{*}$) and of lift-and-project cuts (L\&P).
    We also report the geometric mean of the density of lift-and-project cuts ($\%$nz), measured as the percentage of non-zero coefficients per cut; the density of $\K^{*}$ cuts is not reported because they are always disaggregated.
    Here, $\K^{*}$ cuts refer specifically to cuts obtained from the CGCP that were identified as $\K^{*}$ cuts, i.e., we do not consider those added in the refinement steps.
    For example, the first row of Table \ref{tab:res:sparstity:10rounds} indicates that, for \texttt{clay} instances and after 10 rounds, the CGCP with standard normalization yielded an average of $348.8$ disaggregated $\K^{*}$ cuts, and $67.8$ lift-and-project cuts with an average density of $2.5 \%$.

    \begin{table}
        \centering
        \caption{Cut statistics - 10 rounds}
        \label{tab:res:sparstity:10rounds}
        \begin{tabular}{lrrrrrrrrrrrr}
            \toprule
            & \multicolumn{3}{c}{Alpha} & \multicolumn{3}{c}{Polar} & \multicolumn{3}{c}{Standard} & \multicolumn{3}{c}{Uniform}\\
            \cmidrule(rl){2-4}
            \cmidrule(rl){5-7}
            \cmidrule(rl){8-10}
            \cmidrule(rl){11-13}
            Group  & $\K^{*}$ & L\&P & $\%$nz & $\K^{*}$ & L\&P & $\%$nz & $\K^{*}$ & L\&P & $\%$nz & $\K^{*}$ & L\&P & $\%$nz \\
            \midrule
            \texttt{clay}    &    0.0 &  134.8 &  44.9 &    0.0 &    0.6 &   4.7 &  348.8 &   67.8 &   2.5 &   39.9 &   92.2 &   1.4\\
            \texttt{flay}    &    0.0 &  118.7 &  27.6 &    0.0 &   15.2 &   4.3 &   11.9 &  106.9 &   3.6 &    6.9 &  115.1 &   2.3\\
            \texttt{slay}    &    0.0 &  267.0 &  27.1 &    0.0 &   21.9 &   6.4 &   39.1 &  183.3 &   1.9 &    8.6 &  205.7 &   1.3\\
            \texttt{fmo}     &   31.0 &  244.6 &  35.5 &    5.9 &   12.9 &   8.8 &   96.6 &  235.7 &   1.7 &   63.1 &  229.8 &   1.3\\
            \texttt{sssd}    &    0.0 &  282.1 &  16.5 &    0.0 &   45.2 &   8.0 &  164.5 &  127.8 &   7.4 &  131.2 &  147.3 &   8.0\\
            \texttt{tls}     &  362.1 &  123.7 &  13.3 &   93.1 &  112.5 &   5.8 &   90.9 &   84.1 &   7.2 &   36.1 &  104.2 &   7.0\\
            \texttt{uflquad} &    0.0 &   41.3 &  34.0 &    0.0 &  121.8 &   1.3 &  210.6 &   51.4 &   0.5 &  150.1 &   63.3 &   0.5\\
            \midrule
            All     &    3.3 &  188.4 &  29.7 &    1.3 &   16.6 &   6.2 &   93.9 &  143.2 &   2.5 &   44.1 &  157.1 &   1.9\\
            \bottomrule
        \end{tabular}
    \end{table}
    
    \begin{table}
        \centering
        \caption{Cut statistics - 200 rounds}
        \label{tab:res:sparsity:200rounds}
        \begin{tabular}{lrrrrrrrrrrrr}
            \toprule
            & \multicolumn{3}{c}{Alpha} & \multicolumn{3}{c}{Polar} & \multicolumn{3}{c}{Standard} & \multicolumn{3}{c}{Uniform}\\
            \cmidrule(rl){2-4}
            \cmidrule(rl){5-7}
            \cmidrule(rl){8-10}
            \cmidrule(rl){11-13}
            Group  & $\K^{*}$ & L\&P & $\%$nz & $\K^{*}$ & L\&P & $\%$nz & $\K^{*}$ & L\&P & $\%$nz & $\K^{*}$ & L\&P & $\%$nz \\
            \midrule
            \texttt{clay}    &    0.0 & 2568.8 &  44.9 &    0.0 &    0.8 &   4.6 & 6366.4 & 1237.9 &   3.8 & 3003.0 & 2880.2 &   2.1\\
            \texttt{flay}    &    0.0 & 1950.7 &  24.4 &    0.0 &   77.1 &   5.0 &  382.7 & 1511.9 &   4.1 &  193.4 & 1479.9 &   3.2\\
            \texttt{slay}    &    0.4 & 5486.4 &  21.6 &    0.0 &   43.9 &   6.5 & 2939.0 & 2309.1 &   5.6 &  470.8 & 1760.1 &   4.6\\
            \texttt{fmo}     & 1158.0 & 5059.3 &  33.9 &   39.6 &   59.9 &   9.1 &  458.2 & 2491.2 &   4.1 &  296.4 & 2062.4 &   3.3\\
            \texttt{sssd}    &    0.0 & 5821.0 &  10.4 &    0.0 &   57.1 &   8.2 &  342.2 &  429.0 &  14.2 &  222.4 &  362.3 &  12.6\\
            \texttt{tls}     & 2006.4 &  742.2 &  13.7 &  300.2 &  359.4 &   6.3 &  550.3 &  576.3 &   8.1 &  415.7 &  796.2 &   8.1\\
            \texttt{uflquad} &    0.0 &   41.3 &  34.0 &    0.0 &  285.3 &   1.2 &  410.3 &  129.4 &   0.4 &  344.2 &  153.1 &   0.3\\
            \midrule
            All     &   17.6 & 2894.3 &  26.1 &    3.7 &   44.7 &   6.4 &  756.7 & 1256.5 &   4.7 &  388.5 & 1232.3 &   3.7\\
            \bottomrule
        \end{tabular}
    \end{table}

    First, very few cuts are obtained with the polar normalization, with the exception of \texttt{tls} and \texttt{uflquad} instances.
    As mentioned earlier, this is the result of premature termination of the cut generation due to numerical issues in the CGCP.
    Nevertheless, the obtained cuts are relatively sparse, with only $6.2\%$ and $6.4\%$ non-zeros after 10 and 200 rounds, respectively.
    
    Second, with the exception of \texttt{sssd} instances, the $\alpha$ normalization always yields denser cuts than other normalizations, with an average of $29.7\%$ and $26.1\%$ non-zeros coefficients after 10 and 200 rounds, respectively.
    Denser cuts have an adverse effect on performance, as they slow down the resolution of the current linear relaxation, and are more prone to numerical errors.
    Overall, few $\K^{*}$ cuts are obtained, except for \texttt{fmo} and \texttt{tls} instances.
    
    Third, fewer $\K^{*}$ cuts are obtained with the uniform normalization than with the standard normalization, which is not surprising given the remarks in Section \ref{sec:infeas}.
    Furthermore, cuts obtained with the uniform normalization are slightly sparser, with $1.9\%$ and $3.7\%$ non-zero coefficients after 10 and 200 rounds for the uniform normalization, compared to $2.5\%$ and $4.7\%$ for the standard normalization.
    Interestingly, the average cut density after 200 rounds is almost twice as large as after 10 rounds, indicating that cuts become denser in the later rounds.

\section{Conclusion}
\label{sec:conclusion}

Motivated by the impact of the disjunctive framework in MILP and the recent success of conic formulations for \MICONV, we have investigated the computational aspects of disjunctive cuts in \MICONIC.
Building on conic duality, we have extended Balas' cut-generating linear program into a cut-generating conic program, and studied the fundamental role of the normalization condition in its resolution.
In doing so, we have answered several relevant questions, especially from the numerical standpoint, left open by previous developments in the area, and have raised new ones.

\subsection{What we have learned...}

From a theoretical standpoint, we have shown that the normalization condition in the CGCP impacts not only the theoretical properties of the obtained cuts, but also whether the CGCP is solvable in the first place.
The latter has direct consequences for the numerical robustness of the CGCP resolution.
In particular, we have introduced conic normalizations that guarantee conic strong duality, without any assumptions on the well-posedness of the considered disjunctions.
Furthermore, we have identified the risk of generating $\K^{*}$ cuts when separating conic-infeasible points, and suggested several strategies to alleviate it.
Finally, we have proposed extensions of lifting and cut strengthening to the conic setting, which may provide further computational benefits.

From a computational standpoint, we have investigated the practical behavior of several normalization conditions, on a diverse set of instances.
These experiments indicate that our CGCP-based cut separation is competitive with a state-of-the-art \MICONIC\ solver, being able to close more gap in the early stages.
In particular, the numerical robustness of the standard and uniform normalization translates in faster separation and more gap closed.
Finally, carefully managing the outer approximation, so as to avoid large violations of conic constraints, appears critical to the cut generation performance.

\subsection{... and what lies ahead}

Several theoretical questions are left open.
First, while we only consider linear cutting planes, whether conic cuts can be separated efficiently in the general case, and how to best integrate them within existing \MICONIC\ algorithms, remains an open question.
Second, dominance relations between valid inequalities were introduced in, e.g, \cite{KilincKarzan2015_MinimalValidInequalities,KilincKarzanSteffy2016_SublinearInequalitiesMixed}.
Characterizing optimal solutions of the normalized CGCP in that perspective may offer further insight on the importance of normalization.
Third, a unifying framework that generalizes monoidal strengthening to the \MICONIC\ setting could further improve the practical effectiveness of disjunctive cuts.
Recent developments in duality theory and cut-generating functions for \MICONIC\ could offer the tools for doing so.

Computational experience with cuts in \MICONIC\ remains scarce, and further research in that direction is needed.
Decades of experience in MILP indicate that classes of disjunctive cuts with fast separation rules, e.g., Gomory Mixed-Integer cuts and Mixed-Integer Rounding cuts, yield the greatest computational benefits.
Similar strategies for \MICONIC\ would undoubtedly be beneficial.
In a similar fashion, Fischetti et al. \cite{Fischetti2011} have shown that scaling impacts the cuts yielded by the CGLP with standard normalization: there is no indication that \MICONIC\ should be any different.
Finally, the CGCP form allows to experiment with a number of algorithmic techniques for separating disjunctive cuts, whose practical effectiveness will most likely be problem-specific.

\section*{Acknowledgements}
    The second author was supported by an FRQNT excellence doctoral scholarship, and a Mitacs Globalink research award.
    We thank Pierre Bonami, Andrea Tramontani and Sven Wiese for several helpful discussions on the topic.

\bibliographystyle{spmpsci}      
\bibliography{refs}   

\end{document}